\documentclass[12pt,reqno]{amsart}
\usepackage{amssymb, amsmath, amsfonts, amsthm, graphics}
\usepackage[hmargin=1 in, vmargin = 1 in]{geometry}

\RequirePackage[dvipsnames,usenames]{color}

\usepackage{mathrsfs}
\usepackage{mathtools}
\usepackage{tikz-cd} 
\usetikzlibrary{matrix, calc, arrows} 
\usepackage[hyphens]{url}
\usepackage{breakurl}

\usepackage{hyperref}
\usepackage[all]{xy}
\usepackage{marginnote}

\usepackage[active]{srcltx}
\usepackage{calc,amscd, eucal,ulem,stmaryrd}
\usepackage{alltt}

\usepackage{soul}
\usepackage{hyperref}

\input{kmacros3.sty}
\input{xy}
\xyoption{all}
\usepackage{tikz}
\usetikzlibrary{decorations.markings}
\usetikzlibrary{cd}

\usepackage{mabliautoref}
\usepackage{bm}
\usepackage{comment}

\def\todo#1{\textcolor{Mahogany}%
{\footnotesize\newline{\color{Mahogany}\fbox{\parbox{\textwidth-15pt}{\textbf{todo: } #1}}}\newline}}

\newcommand{\isom}{\cong} 

\theoremstyle{definition}
\numberwithin{equation}{section}

\newcommand{\vol}{\mathrm{vol}}
\newcommand{\Amp}{\mathrm{Amp}}
\newcommand{\Nef}{\mathrm{Nef}}

\newcommand{\NN}{\mathbb{N}}

\newcommand{\QQ}{\mathbb{Q}}
\newcommand{\RR}{\mathbb{R}}

\newcommand{\cC}{\mathcal{C}}

\newcommand{\cF}{\mathcal{F}}
\newcommand{\cG}{\mathcal{G}}

\newcommand{\cL}{\mathcal{L}}
\newcommand{\cN}{\mathcal{N}}

\newcommand{\fm}{\mathfrak{m}}
\newcommand{\fn}{\mathfrak{n}}

\newcommand{\Fe}{F^{e}_{*}}
\newcommand{\effcone}{\mathrm{Eff}}
\newcommand{\bigcone}{\mathrm{Big}}

	\title{The $F$-signature function on the ample cone}

    \author{Seungsu Lee}
    \address[S.~Lee]{Department of Mathematics, University of Utah, Salt Lake City, UT 84112, USA}
    \email{\href{mailto:slee@math.utah.edu}{slee@math.utah.edu}}
    \thanks{Lee was supported in part by NSF Grant  \#2101800 and  \#1952522}
    
    \author{Suchitra Pande}
	\address[S.~Pande]{Department of Mathematics\\University of Michigan\\Ann Arbor, 
		MI 48109-1043\\USA}
	\email{\href{mailto:swarajsp@umich.edu}{swarajsp@umich.edu}}
    \thanks{Pande was partially supported by the NSF grants \#1952399, \#1801697 and \#2101075}
    
\begin{document}
 
	\begin{abstract}
	   For any fixed globally $F$-regular projective variety $X$ over an algebraically closed field of positive characteristic, we study the $F$-signature of section rings of $X$ with respect to the ample Cartier divisors on $X$. In particular, we define an $F$-signature function on the ample cone of $X$ and show that it is locally Lipschitz continuous. We further prove that the $F$-signature function extends to the boundary of the ample cone. We also establish an effective comparison between the $F$-signature function and the volume function on the ample cone. As a consequence, we show that for divisors that are nef but not big, the extension of the $F$-signature is zero. 
	\end{abstract}
	\maketitle

	\section{Introduction}

 The \emph{$F$-signature} is an invariant of the singularities of a Noetherian, $F$-finite local ring $R$ of prime characteristic. First arising implicitly in \cite{SmithVanDenBerghSimplicityOfDiff} and formally defined in \cite{HunekeLeuschkeTwoTheoremsAboutMaximal}, this invariant measures the asymptotic growth of the number of \emph{Frobenius splittings} of $R$ (See \autoref{F-sigdfn}). The positivity of the $F$-signature corresponds exactly to $R$ being a strongly $F$-regular singularity \cite{AberbachLeuschke}; only when $R$ is regular, does the $F$-signature achieve its maximum value of $1$. This invariant has also attracted attention as a candidate for the positive characteristic analog of the \emph{normalized volume} of a Kawamata log-terminal (klt) singularity, extending the established analogy between strongly $F$-regular and klt singularities; see \cite{LiLiuXuAGuidedTour}, \cite{TaylorInvAdjFsig}, \cite{MaPolstraSchwedeTuckerFsigBirational}. There have been applications of the $F$-signature to bounding the sizes of the \'etale fundamental group and the torsion subgroup of the divisor class group; see \cite{carvajal-rojas_fundamental_2016}, \cite{MartinTorsionDivisors}, \cite{CarvajalRojasFiniteTorsors} and \cite{PolstramaximalCM}.

    
    In the global setting, \emph{globally $F$-regular}  varieties (\autoref{defn.GlobFreg}), introduced in \cite{SmithGloballyFRegular} are the positive characteristic analogs of \emph{log-Fano} type varieties enjoying additional properties such as satisfying a \emph{Kawamata-Viehweg vanishing} theorem \cite{SchwedeSmithLogFanoVsGloballyFRegular}.  It was shown in \cite{SmithGloballyFRegular} that a projective variety $X$ is globally $F$-regular if and only if the $\emph{section ring}$ $S(X, \cL)$ (\autoref{sectionringdfn}) is strongly $F$-regular for some (equivalently, every) ample invertible sheaf $\cL$. Since the $F$-signature is positive for all strongly $F$-regular rings, it is natural to ask: \emph{How does the $F$-signature of the section ring $S(X,\cL)$ vary with $\cL$?} The purpose of this paper is to answer this question. We prove:

	\begin{thm}\label{mainthm1intro}
	Fix any globally $F$-regular projective variety $X$ over an algebraically closed field $k$ of positive characteristic $p$. Assume that the dimension of $X$ is positive. Then, the $F$-signature function $ L \mapsto s_{X}(L)$, assigning to any ample Cartier divisor $L$, the $F$-signature of the section ring of $X$ with respect to $L$, satisfies the following properties:
	\begin{enumerate}
	    \item (\cite{VonKorffthesis}, \cite{CarvajalRojasFiniteTorsors},  \autoref{propertiesofs}) The $F$-signature function $s_{X}$ naturally extends to a unique, well-defined, real-valued function
	    $$ s_{X}: \Amp_\QQ(X) \to \RR $$
	    on the set of rational classes in the ample cone of $X$ satisfying:
	    $$ s_{X}(\lambda L) = \frac{1}{\lambda} s_{X}(L) \quad \text{for all ample $\QQ$-divisors $L$ and all $\lambda \in \QQ_{>0}$.}    $$
	    \item (\autoref{continuityatrational}) The function $s_X$ is continuous on the rational ample cone of $X$, with respect to the usual topology on the N\'{e}ron-Severi space.
	    \item (\autoref{realcor}) The function $s_{X}$ extends continuously to all real classes in the ample cone of $X$.
	\end{enumerate}
	\end{thm}
	
	\autoref{mainthm1intro} provides us with a new tool for the study of globally $F$-regular varieties. Such varieties have found various applications, for instance, to the three dimensional minimal model program in positive characteristic \cite{HaconXuThreeDimensionalMinimalModel} and in the study of Fano type complex varieties \cite{GongyoSingularitiesofcoxrings}. For other investigations regarding globally $F$-regular varieties, see \cite{gongyo_rational_2015}, \cite{GongyoTakagiInjectivityThmforGFR} and \cite{KawakamiBogomolovSommese}.
	
	 Another motivation for considering the $F$-signature function comes from the \emph{volume function} on the big cone of a projective variety. On a projective variety $X$ over an algebraically closed field, to any Cartier divisor $D$ on $X$, we can associate a non-negative real number called the \emph{volume of $D$}, measuring the growth of the global sections of multiples of $D$. A foundational result in the theory of volumes is that the volume of a big divisor $D$ depends only on its numerical equivalence class. Moreover, it extends suitably to all $\RR$-divisors and varies continuously as $D$ varies on the N\'eron-Severi space of $X$. See \cite[Section 2.2]{LazarsfeldPositivity1} and \cite{LazarsfeldMustataConvexbodies} for the details. The study of volumes of divisors has been important in birational geometry; for example, see \cite[Section 2.2]{LazarsfeldPositivity1}, \cite{LazarsfeldMustataConvexbodies}, \cite{BoucksomVolumeofLineBundle}, \cite{EinLazMusNakPopAsymptoticInvariantsofLineBundles}, \cite{HaconMckernanBoundednessofpluricanonicalmaps}, \cite{TakayamaPluricanonicalsystems}, and \cite{KuronyaAsymptoticcohomologicalfunctions}. 
	 
	 \autoref{mainthm1intro} parallels the theory of the \emph{$F$-signature function} (\autoref{F-sigfunction}) on the ample cone of a globally $F$-regular variety with the volume function on the big cone. This perspective was first considered in \cite{VonKorffthesis}, where \autoref{mainthm1intro} was proved in the special case when $X$ is a toric variety.

    
   

    

    The ample cone is an open cone in the N\'eron-Severi space of a projective variety $X$, and its closure is represented by the set of nef divisors on $X$. Hence, it is natural to ask if the $F$-signature function $s_{X}$ from \autoref{mainthm1intro} has a natural extension to the nef cone. We show that this is indeed true:

   \begin{thm}[\autoref{extensionthm}] \label{extensionthmintro}
       Suppose $X$ is a globally $F$-regular projective variety. Then the $F$-signature function $s_{X}$ extends continuously to all non-zero classes of the Nef cone of $X$. Moreover, if $L$ is a nef Cartier divisor which is not big, then $s_{X}(L) = 0$.
   \end{thm}

    The proof of \autoref{mainthm1intro} and \autoref{extensionthmintro} consists of several steps. First, we need to verify that the $F$-signature function is well-defined on the rational ample cone of $X$. We do this in Section 3. The main result here is that on globally $F$-regular projective varieties, numerical equivalence, and $\QQ$-linear equivalence coincide. This is reminiscent of the same fact for log-Fano type varieties over the complex numbers. Once we have this, we prove the continuity of the $F$-signature function in Section 4. This needs several ideas towards analyzing \emph{Frobenius splittings} of linear systems; a sketch of the proof is presented in \autoref{sketchofproof}. We further utilize these ideas to extend the $F$-signature function to all non-zero nef divisors in Section 5. Lastly, in \autoref{localbounds} we prove a local effective upper bound for the $F$-signature function.

\vskip 12pt

Throughout this paper, unless specified otherwise, all rings are assumed commutative with a unit and are of positive characteristic $p$. $k$ will denote an algebraically closed field of characteristic $p$. All varieties are assumed to be integral, separated schemes of finite type over $k$.
 
	\section{Preliminaries}

	\subsection{$F$-signature}
	
 Let $R$ be any ring of prime characteristic $p$. Then $R$ is naturally equipped with the \emph{Frobenius morphism}, $F: R \to R$ sending $r \mapsto r^p$.  Since $R$ has characteristic $p$, $F$ defines a ring homomorphism, allowing us to view $R$ as a new $R$-module obtained via restriction of scalars along $F$. We denote this new $R$-module by $F_{*} R$ and its elements by $F_{*} r$ (where $r$ is an element of $R$). Concretely, $F_{*} R$ is the same as $R$ as an abelian group, but the $R$-module action is given by:
 $$ r\cdot F_*s := F_{*} r^{p}s \textrm{\quad for  $r\in R$ and $F_*s \in F_*R$}   .$$
 Throughout, we will assume that $R$ is essentially of finite type over $k$, which also makes it  $F$-finite, i.e., $F_{*} R$ is a finitely generated $R$-module. Similarly, for any natural number $e \geq 1$, we have the iterate of the Frobenius, $F^{e} : R \to R$ sending $r \mapsto r^{p^{e}}$ and the $R$-module $\Fe R$ obtained by restricting scalars along $F^e$. We will be interested in invariants of $R$ defined by analyzing the $R$-module structure of the modules $\Fe R$.
 
 	
 	\begin{dfn}[Free rank] \label{frrkkdfn}
 	Let $M$ be a finitely generated module over a local ring $R$. Consider a decomposition:
 	$$ M \isom R^{a(M)} \oplus N      $$
 	where $N$ has no $R$-summands. Then, since $R$ is local, the number $a(M)$ is independent of the decomposition chosen, and is called the \emph{free rank} of $M$. 
 	\end{dfn}

	\begin{dfn}[$F$-signature] \label{F-sigdfn}
 		Let $(R, \fm ,k)$ be a local ring, and $a_e(R)$ denote the free rank of $\Fe R$ (\autoref{frrkkdfn}). Then the \emph{$F$-signature} of $R$ is defined to be the limit:
 		$$ s(R) := \lim_{e \to \infty} \frac{a_{e}(R)}{p^{ed}} $$
 		where $d$ is the Krull dimension of $R$. This limit exists by \cite{TuckerFSigExists}.
 	\end{dfn}
 	
 	The $F$-signature  of $R$ admits an alternate description as follows: We say a map of $R$-modules $\phi: M \to N$ \emph{splits} if there exists an $R$-module map $\psi: N \to M$ such that $\psi \circ \phi = \text{id}_{M}$. Define the subset $I_e \subseteq R$ as 
 	\[I_e = \left\{x\in R \mid \text{the map } R \to F^e_*R \text{ sending } 1\mapsto F^e_* x \text{ does not split}\right\}.\]
 		Then, we observe that $I_e$ is an ideal of $R$
 	and by \cite[Proposition 4.5]{TuckerFSigExists}, the free rank of $\Fe R$ equals $l_R(R/I_e)$, the length of the $R$-module $R/I_{e}$. Hence, the $F$-signature of $R$ can be defined as the limit:
 	$$    s(R) = \lim _{e \to \infty} \frac{l_{R}(R/I_{e})}{p^{ed}} . $$
 	
 	Though the definition of $F$-signature is given for a local ring $(R, \fm, k)$, we may also work with $\NN$-graded rings $(S, \fm ,k)$ i.e. $S$ is $\NN$-graded with $S_{0} = k$ and $\fm = S_{>0}$. We next relate the local and graded situations.
 
	\begin{dfn}[Graded free rank]
	    	Let $(S, \fm, k)$ be an $\NN$-graded ring, finitely generated over $k$, with $S_{0} = k$ and $M$ a finitely generated $\ZZ$-graded module over $S$. Then we can decompose $M$ as a graded $S$-module as:
	    	$$      M \isom P \oplus N $$
	    	where $P$ is a graded free $S$-module (i.e. a direct sum of $S(j)$, the shifted rank $1$ free modules, for various $j \in \ZZ$) and $N$ is a graded module with no graded free summands. Then the rank of $P$ is independent of the chosen decomposition and we define it to be the \emph{graded free rank} of $M$ over $S$ (denoted by  $a_{\mathrm{gr}}(M)$).
	\end{dfn}

 		\begin{lem} \label{fvsgf}
 		Let $(S, \fm, k)$ and $M$ be as above. Then the free rank of $M_{\fm}$ over the local ring $S_{\fm}$ is the same as the graded free rank of $M$.
        \end{lem}
        \begin{proof}
            See \cite[Proposition 5.7]{DeStefaniPolstraYaoGlobalFsplittingRatio} \qedhere
        \end{proof}


        Now, we describe the $F$-signature of $\NN$-graded rings, relating it to the (local) $F$-signature \emph{at the vertex}. For similar discussions relating the local and global situations, see \cite[Section 3]{SmithGloballyFRegular}, \cite[Section 4]{SmithVanishingSingularitiesAndEffectiveBounds}, and \cite[Section 2.2]{VonKorffthesis}.
        
        Let $S$ be an $\NN$-graded ring. Then $\Fe S$ is also naturally an $\frac{1}{p^{e}} \NN$-graded $S$-module by taking
        $$  (\Fe S) _{\frac{i}{p^{e}}}    = \Fe S_{i}. $$ This gives rise to the $\NN$-grading on $\Fe S$ given by
        $$ (\Fe S)_{n} = \bigoplus  _{0 \leq i \leq p^e -1} (\Fe S)_\frac{i+np^e}{p^{e}}     .$$
        Thus,
        $\Fe S$ decomposes as
        $$  \Fe S = \bigoplus_{0 \leq i \leq p^e -1} \bigoplus _{j \geq 0} \Fe S_{i + j p^e}    $$
        as an $\NN$-graded $S$-module. 
        
\begin{dfn}[$F$-signature of $\NN$-graded rings] \label{Fsigofgradedrings}
      Let $(S, \fm,k)$ be an $\NN$-graded, finitely generated $k$-algebra, with $S_{0} = k$. Then, we define the $F$-signature of $S$ to be the limit:
      $$ \lim _{e \to \infty} \frac{a_{e, \textrm{gr}}(S)}{p^{ed}}        $$
      where $a_{e, \text{gr}}(S)$ is the graded free rank of $\Fe S$ and $d$ denotes the Krull dimension of $S$.
      We note that by \autoref{fvsgf}, the $F$-signature of $S$ coincides with the $F$-signature of $S_{\fm}$, the localization of $S$ at the maximal ideal $\fm$.
\end{dfn}

\subsection{Section Rings:}
  	The $\NN$-graded rings we will be interested in arise as the section rings of projective varieties over $k$ with respect to some ample divisor.
  	
 		\begin{dfn}[Section Rings and Modules] \label{sectionringdfn}
		Let $X$ be a projective variety over $k$, $\cL$ an ample invertible sheaf on $X$ and $\cF$ a coherent sheaf on $X$. Then the $\NN$-graded ring $S$ defined by
		$$ S =	S(X, \cL) :=  \bigoplus _{n \geq 0} H^{0}(X, \cL^{n})	$$
		is called the \emph{section ring} of $X$ with respect to $\cL$. The affine scheme $\Spec(S)$ is called the \emph{cone over $X$} with respect to $\cL$. The \emph{section module} of $\cF$ with respect to $\cL$ is a $\ZZ$-graded $S$-module $M$ defined by
		$$M = M(X, \cL) := \bigoplus _{n \in \ZZ} H^{0}(X, \cF \otimes \cL^{n}) 	.$$
		Similarly, the sheaf corresponding to $M$ on $\Spec(S)$ is called the \emph{cone over $\cF$} with respect to $\cL$.
	\end{dfn}

In the next lemma, we record some useful principles concerning direct summands  of sheaves on a proper variety.

\begin{lem}\label{lem.directsummand}
    On a proper variety $X$ over $k$, let $\cL$, $\cM$ be invertible sheaves, and $\cF$, $\cG $ be coherent sheaves. Then,
    \begin{enumerate}
        \item If $\cL$ is not a direct summand of $\cF$ and $\cG$, then $\cL$ is also not a summand of $\cF \oplus \cG$.
        \item If $\cF \isom \cL^{\oplus n} \oplus \cG$ and $\cL$ is not a summand of $\cG$, then, $n$ is the maximum number of $\cL$ summands of $\cF$ (in any decomposition).
        \item Assume $\cL \not\cong \cM$, and both $\cL$ and $\cM$ are summands of $\cF$, then $\cL \oplus \cM$ is a summand of $\cF$.  
    \end{enumerate}
\end{lem}

\begin{proof}
 Note that an $\cO_X$-summand (i.e., a summand isomorphic to $\cO_X$) of a coherent sheaf $\cF$ is equivalent to a non-zero global section $ s \in H^0 (X, \cF ) $ and a map $ \varphi \in \Hom_{\cO_X} (\cF ,  \cO_X) $ such that $\varphi(s) \neq 0$.
    \begin{enumerate}
        \item By twisting by $\cL^{-1}$, we may assume that $\cL = \cO_X$. An $\cO_X$-summand of $\cF \oplus \cG$ is given by a global section $$ s = (s_1, s_2) \in H^0 (X, \cF \oplus \cG) =  H^{0}(X, \cF) \oplus H^0 (X, \cG)$$ and a map
        $$ \varphi = (\varphi_1, \varphi_2) \in \Hom_{\cO_X} (\cF \oplus \cG,  \cO_X) = \Hom_{\cO_X} (\cF, \cO_X) \oplus \Hom_{\cO_{X}} (\cG, \cO_X) $$
        such that $\varphi(s) \neq 0$. However,
        $ \varphi(s) = \varphi_1 (s_1) + \varphi_2 (s_2) $. So, if $\varphi(s) \neq 0$, then $\varphi_{i} (s_i) \neq 0$ for some $i =1,2$, giving an $\cO_X$-summand of either $\cF$ or $\cG$, which is a contradiction.

        \item Again, twisting by $\cL^{-1}$, we may reduce to the case when $\cL = \cO_X$. Suppose that there is another decomposition $\cF \cong \cO_X^{\oplus (n + m)} \oplus \cG'$ for some $m>0$. Let $\varphi : \cO_X^{\oplus (n + m)} \bigoplus \cG' \to \cO_X^{\oplus n}\bigoplus \cG$ be an isomorphism. Now, consider the map $\psi: H^0 (X, \cO_{X} ^{\oplus (n+m)} ) \to H^0 (X, \cO_X ^{\oplus n})$ induced by the inclusion of $\cO_X ^{\oplus (n+m)}$ into $\cF$, the isomorphism $\varphi$ and the projection onto $\cO_{X} ^{\oplus n}$. Since $m$ is positive, there exists a non-zero section $s \in \cO_X^{\oplus (n+m)}$ such that $\psi (s) = 0$.
        
        Now write $\varphi (s, 0) = (0, g)$ for some $g \in H^0 (X, \cG)$. Note that $(s,0)$ gives an $\cO_{X}$-summand of $\cF$. Hence, $g$ must be an $\cO_X$-summand of $\cG$, which is a contradiction, since $\cG$ was assumed to have no $\cO_X$-summands.

        \item Since $\cL$ is a direct summand of $\cF$, there is some $\cG$ such that $\cF \cong \cL \oplus \cG$. Now, by part (a), if $\cM$ is a direct summand of $ \cL \oplus \cG$, then $\cM$ is direct summand of either $\cL$ or $\cG$. However, since $\cM \not\cong \cL$, $\cM$ must be a direct summand of $\cG$. Hence, $\cL \oplus \cM$ is a direct summand of $\cF$.  \qedhere
    \end{enumerate}
\end{proof}

\subsection{$F$-regularity:}

\begin{dfn}[Strong $F$-regularity] \cite{HochsterHunekeTightClosureAndStrongFRegularity}
      Let $R$ be a Noetherian $F$-finite ring of characteristic $p$. Then, $R$ is said to be \emph{strongly $F$-regular} if for any element $c \in R$ that is not contained in any minimal prime of $R$, there exists an integer $e \gg 0$, such that, the following map
      \begin{equation*}
      \begin{split}
          R \to \Fe R \\
          1 \mapsto \Fe c
      \end{split}
      \end{equation*}   
      splits as a map of $R$-modules.
\end{dfn}

\begin{dfn}[Global $F$-regularity]\label{defn.GlobFreg} \cite[Definition 3.2]{SchwedeSmithLogFanoVsGloballyFRegular} Let $X$ be a normal variety over $k$. Then $X $ is said to be \emph{globally $F$-regular} if for any effective Weil divisor $D$ on $X$, there exists an integer $e \gg 0$, such that, the natural map 
\begin{equation*} 
\cO_X \to F^e_* \cO_X(  D)
 \end{equation*}
splits as a map of $\cO_{X}$-modules.

\end{dfn}

\begin{rem} When $X= \Spec(R)$ is an affine variety, $X$ being globally $F$-regular is equivalent to $R$ being strongly $F$-regular \cite{SmithGloballyFRegular}.

\end{rem}

\begin{rem}\label{FregularityremarK}
A local ring $R$ is strongly $F$-regular if and only if its $F$-signature $s(R)$ is positive \cite{AberbachLeuschke}.
\end{rem}

\begin{thm} \cite[Theorem 3.10]{SmithGloballyFRegular} 
    Let $X$ be a projective variety over $k$. Then, $X$ is globally $F$-regular if and only if the section ring $S(X, \cL)$ (\autoref{sectionringdfn}) with respect to some (equivalently, every) ample invertible sheaf $\cL$  is strongly $F$-regular.
\end{thm}

Combining with \autoref{FregularityremarK}, $X$ is globally $F$-regular if and only if  the $F$-signature $s(S(X, \cL)) $ is positive for some (equivalently,
every) ample invertible sheaf $\cL$ on $X$.

\begin{thm}[\cite{SmithGloballyFRegular}, Corollary 4.3] \label{vanishingforGFR}
Let $X$ be a projective, globally $F$-regular variety over $k$. Suppose $\cL$ is a nef invertible sheaf over $X$. Then,
$$  H^{i}(X, \cL) = 0 \quad \text{for all $i >0$}.   $$
\end{thm}

\section{Definition of the $F$-signature Funtion}

  In this section, we will define an $F$-signature function on the rational ample cone of a globally $F$-regular projective variety.  The \emph{rational ample cone}, consisting of numerical equivalence classes of ample $\QQ$-divisors on $X$ will be denoted by $\Amp_{\bQ} (X)$. Recall that in the N\'eron-Severi space $N^1 _{\RR} (X)$, $\Amp _{\QQ}(X)$ is the set of rational points of the open cone $\Amp_{\RR}(X)$ (which consists of classes of ample $\RR$-divisors on $X$). Hence, $\Amp_{\QQ} (X)$ has a natural topology, inherited from any norm on $\text{N}^1 _{\RR} (X)$. We refer to \cite[Chapter 1]{LazarsfeldPositivity1} for the details.
 
 \begin{dfn} \label{F-sigfunction}
 Let $X$ be a globally $F$-regular projective variety over $k$ (\autoref{defn.GlobFreg}). Suppose that $\dim (X) >0$. The $F$-signature function
 $$ s_X: \text{Amp}_{\QQ}(X) \to \RR    $$
on the rational ample cone of $X$, is defined as follows:
 \begin{enumerate}
     \item If the class $[L] \in \text{Amp}_{\QQ}(X)$ is defined by an integral Cartier divisor $L$, then we define $s_{X}([L])$ to be the $F$-signature (\autoref{F-sigdfn}) of the section ring $S(X, L)$ (\autoref{sectionringdfn}) of $L$:
     $$ s_X([L]) := s(S(X,L)).    $$
     \item If the class $[L]$ is defined by a rational multiple of an integral Cartier divisor i.e. $L = \frac{a}{b} D$ where $D$ is an integral Cartier divisor on $X$, then we define:
     $$ s_X([L]) : = \frac{b}{a} s_X([D]) = \frac{b}{a} s(S(X,D))  .  $$
 \end{enumerate}
 \end{dfn}
 The rest of this section is devoted to checking that the function $s$ is indeed well-defined.

\begin{rem}
If $\dim (X) = 0$, we define the $F$-signature function as $s_X(L) = 1$ for any ample divisor on $X$. Indeed, $X$ is just a point and the only divisor on $X$ is 0. 
\end{rem}
 
 \begin{thm} \label{propertiesofs}
Let $X$ be a globally $F$-regular projective variety over $k$ with $\dim(X)$ positive. Then, \autoref{F-sigfunction} gives a well-defined $F$-signature function $s_X$ on the rational ample cone of $X$, satisfying the identity:
 $$ s_X\left(\frac{a}{b}L\right) = \frac{b}{a}s_X(L)        $$
 for any two non-zero natural numbers $a$ and $b$ and any ample $\QQ$-divisor $L$.
  \end{thm}
 \begin{proof}

 To prove the theorem, we need to check that the function $s_X$ as defined in \autoref{F-sigfunction} is well-defined. There are two issues:
 \begin{enumerate}
     \item The first arising from the choice of a $\QQ$-divisor representing a numerical equivalence class (\autoref{numericalequivalence}).
     \item Having chosen a $\QQ$-divisor $L$ representing a numerical class, there is still ambiguity in choosing a representation of $L$ as a rational multiple of an integral Cartier divisor (\autoref{scalingwithoutpairs}).
 \end{enumerate}
 
 We address the first ambiguity by proving that on a globally $F$-regular variety, numerical equivalence and linear equivalence are the same conditions. This is an analog of the same result for \emph{log-Fano} varieties over the complex numbers, a well-known consequence of the \emph{Kawamata-Viehweg} vanishing theorem. The following theorem maybe well-known to experts, but we do not know a reference.
 
 \begin{thm} \label{numericalequivalence}
 Let $X$ be a projective, globally $F$-regular variety over $k$. Suppose $\cL$ is a numerically trivial invertible sheaf on $X$, i.e. $\deg(\cL|_{C}) = 0$ for all curves $C$ on $X$. Then, $\cL$ is isomorphic to the trivial invertible sheaf $\cO_{X}$.
 \end{thm}
 
 \begin{proof}
 First, we note that by \cite[Ch. 2, Section 2, Corollary 1]{Kleimannumericaltheory}, some power $\cL^{m}$ of $\cL$ is algebraically equivalent to $\cO_{X}$ i.e. $\cL^{m}$ deforms to $\cO_{X}$. Since the Euler-characteristic (for sheaf-cohomology) is invariant in flat families \cite[Ch. III, Theorem 9.9]{Hartshorne}, we get that 
 \begin{equation}
     \chi(\cO_{X}) = \chi(\cL^{m})
 \end{equation}
 for some natural number $m$. Now, by \autoref{vanishingforGFR}, since $\cO_{X}$ and $\cL^{m}$ are both nef invertible sheaves, we get that 
 \begin{equation}
     H^{i}(X, \cO_{X}) = H^{i} (X, \cL^{m}) = 0 \quad \text{for all $i >0$.}
 \end{equation}
 Hence, we get
 \begin{equation}
     1 = h^{0}(X, \cO_{X}) = \chi(\cO_{X}) = \chi(\cL^{m}) = h^{0}(X, \cL^{m}) .
 \end{equation}
 Hence, we have shown that $\cL^{m}$ has a non-zero global section. But, since it is also numerically trivial, it must indeed be trivial (since an effective divisor cannot be numerically trivial unless it is the zero divisor). Therefore, $\cL^{m} \isom \cO_{X}$.
 
 Now, by \cite[Ch. 1, Section 1]{Kleimannumericaltheory}, we have that the function $\chi(\cL^{n})$ is a polynomial function of $n$ (as $n$ varies over all integers). Since $\cL^{m} \isom \cO_{X}$, we must have that $\chi(\cL^{n}) = 1$ for all $n \in \ZZ$. But, again, since $\cL^{n}$ is nef for all $n \geq 0$, by \autoref{vanishingforGFR} we have
 \begin{equation}
      h^{0}(X, \cL) = \chi(\cL) = 1.
 \end{equation}
 Hence, $\cL \isom \cO_{X}$ as well because $\cL$ is numerically trivial and has a non-zero global section.
 
 \end{proof}
 
 \begin{rem}
 It was proved in \cite{CarvajalRojasFiniteTorsors} that torsion divisors (i.e., $L$ such that $nL \sim 0$ for some $n$) are themselves linearly equivalent to $0$. Hence, the last part of the proof above follows from this fact, but we include a proof for the convenience of readers.
 \end{rem}
 
 Next, we address the second kind of ambiguity in \autoref{F-sigfunction}. For this, we note the following scaling property for $F$-signature of section rings under taking Veronese subrings, first observed in \cite{VonKorffthesis}.
 
 \begin{thm} \cite[Theorem 2.6.2]{VonKorffthesis}  \label{scalingwithoutpairs}
Let $X$ be a projective variety over $k$ with $\dim (X)$ positive and $\cL$ an ample invertible sheaf on $X$. Let $S(\cL)$ and $S(\cL^{n})$ denote the section rings with respect to $\cL$ and $\cL^{n}$ respectively, where $n$ is any positive natural number. Then, we have the following relation between their $F$-signatures:
\begin{equation}\label{eq.Scalingwithoutpairs}
    s(S(\cL)) = n \, s(S(\cL^{n})).
\end{equation}
\end{thm}

 \end{proof}



 \section{Continuity of the $F$-signature function}
In this section, we prove that the $F$-signature function (\autoref{F-sigfunction}) varies continuously on the ample cone. Throughout, we fix a globally $F$-regular projective variety $X$ over $k$.

  \begin{thm} \label{continuityatrational}
  The $F$-signature function is continuous at each rational class in the ample cone of $X$.
  \end{thm}
  
  In fact, much more is true: the $F$-signature function is {\it locally Lipschitz} around any real class in the ample cone $\Amp _\RR (X)$, with respect to any norm chosen on 
  the N\'{e}ron-Severi space. More precisely, we prove:
  
  \begin{thm} \label{mainthm} Fix any norm $\| \, \|$ on the N\'{e}ron-Severi space N$^1 _\RR (X)$ of a projective globally $F$-regular variety $X$. Then for each real class 
   $D \in \Amp_\RR(X) $, there exist positive real numbers $C(D)$ and $r(D)$ (depending only on $D$ and the norm $\| \, \|$), such that for any two ample $\QQ$-divisors $L, \, L^{\prime}$ contained in the ball $B_{r(D)}(D) := \{D^\prime \in \Amp_{\RR}(X) \,| \, \|D-D^\prime\| < r(D) \}$, we have
  \begin{equation} \label{Lipschitzinequality}
      | s_X(L) - s_X(L^{\prime})| \leq C(D) \| L - L^{\prime} \|.
  \end{equation}
  \end{thm}
  
  We will say that the $F$-signature function $s$ is locally Lipschitz at a real class $D$ with Lipschitz constant $C(D)$ if the inequality~(\ref{Lipschitzinequality}) is satisfied for all ample $\QQ$-divisors $L, \, L^{\prime}$ that are sufficiently close to $D$.

  As an immediate corollary of \autoref{mainthm}, we obtain:

  \begin{Cor} \label{realcor}
Let $X$ be a projective variety over $k$ with $\dim (X)$ positive. Then, the $F$-signature function $s_X$ extends to a well-defined, continuous, locally Lipschitz function on the real ample cone $\Amp_{\RR}(X)$ of $X$ satisfying the identity:
 $$ s_X(\lambda L) = \frac{1}{\lambda} s_X(L) \quad \text{for all $\lambda \in \RR_{>0}$ and all $L \in \Amp_{\RR}(X)$.} $$
  \end{Cor}
\begin{proof}
    Let $D \in \Amp_\RR(X)$ be a real ample class on $X$. The Lipschitz inequality (\ref{Lipschitzinequality})  implies that for any sequence of ample $\QQ$-divisors $L_{n}$ converging to $D$, the sequence $s_X(L_{n})$ is Cauchy, hence converges to a unique real number. This gives a well-defined extension of $s_X$ to the real ample cone $\Amp_\RR(X)$, that remains locally Lipschitz. Hence, $s$ is continuous on $\Amp_{\RR}(X)$. Finally, the identity $s_X(\lambda L) = \frac{1}{\lambda} s_X(L)$ follows by continuity, since it already holds for all rational $L$ and $\lambda$.
\end{proof}

 \subsection{Informal sketch of the proof of \autoref{mainthm}:}

The proof of \autoref{mainthm} consists of several steps. We summarize the ideas in this subsection.


\begin{itemize}
    \item [Step 1:] First, in \autoref{formulaforFsig}, we prove a formula for calculating the $F$-signature of an ample Cartier divisor $L$, in terms of \emph{Frobenius splittings} of the linear systems $|mL|$ for $m \gg 0$. This gives us a tool to compare $s_X(L)$ and $s_X(L^\prime)$ whenever we have a non-zero map $\cO_{X}(mL) \to \cO_{X}(mL^\prime)$ for $m \gg 0$ (\autoref{inclusion}). 
    
    \item [Step 2:] Given two ample $\QQ$-divisors $L$ and $L^{\prime}$, we first consider the case when $L^{\prime} - L$ is big. Since $L^\prime - L$ is big, for $m \gg 0$, we have $|mL^\prime - mL| \neq \phi $ allowing us to compare $s_X(L)$ and $s_X(L^\prime)$ (\autoref{inclusion}). Further, we may find a constant $\alpha$ such that $\alpha L - L^\prime$ is big as well. This allows us a reverse comparison between $s_X(\alpha L)$ and $s_X(L^\prime)$. (\autoref{factorization}). 
    
    \item [Step 3:]  In this step, we estimate the difference in the $F$-signatures by comparing it to the difference in volumes. Here, we encounter the key difficulty, which is that we don't know the sign of the difference between $s_X(L^{\prime})$ and $s_X(L)$, even if $L^{\prime} - L$ is effective, which we have already assumed. This is overcome by introducing the difference between $s_X(L)$ and $s_X(\alpha L)$ (where $\alpha$ is as in Step 3), along with comparisons to the volume function to estimate the difference between $s_X(L)$ and $s_X(L^{\prime})$. These estimates are the contents of \autoref{maininequalitylem} and \autoref{colon}.
    
    \item [Step 4:] To control the difference in the volumes (from Step 4), we need an additional ingredient: For any $e \geq 1$, we need effective bounds for the degrees $m$ that contribute Frobenius splittings to the $e^\text{th}$ free-rank for $S(X,L)$ and $S(X, L^\prime)$ (\autoref{lem.Biglowerdim}).
    
    \item [Step 5:] The steps so far give us an inequality of the form
    \begin{equation} \label{Lipschitzidea} |s_X(L)-s_X(L^\prime)| \leq C(L) \|L- L^\prime \| \end{equation} for a fixed $L$ and all $L^\prime$ sufficiently close to $L$ and for some constant $C(L)$ depending on $L$ (\autoref{ampledifference}).  One result required here is the (Lipschitz) continuity of the volume function on the ample cone (\autoref{volumeLipscitz}).

    \item [Step 6:] Though (\ref{Lipschitzidea}) proves continuity of $s_X$ at a fixed $\QQ$-divisor $L$, it does not prove that $s_X$ is locally Lipschitz, since the constant $C(L)$ depends on $L$. So, in \autoref{explicitLipschitz} and \autoref{uniformconstants}, we track the constant $C(L)$ and examine the variation with $L$. This involves carefully choosing the scalar $\alpha$ from Step 3.
    
    \item [Step 7:]  As a result, we see that we may pick the constants $C(L)$ such that $C(L) = o(\frac{1}{\|L\|^2})$ as $\|L \| \to \infty$. Now, since $s_X(r L) = \frac{1}{r} s_X(L)$ by \autoref{propertiesofs}, we see that for a $\QQ$-divisor $L$, we may pick $C(L) = r^2 C(rL)$ for any $r \gg 0$. This shows that we may pick uniform Lipschitz constants on compact subsets of the ample cone.

    \item [Step 8:] Given two ample $\QQ$-divisors $L$ and $L^\prime$, we may consider a small perturbation $\lambda L^\prime$ of $L^\prime$ (i.e. $\lambda \approx 1$) so that $\lambda L^\prime - L$ is big (or even ample). Using the transformation rule as in \autoref{propertiesofs}, we may replace $L^\prime$ by $\lambda L^\prime$ and reduce to the case when  $L^\prime - L$ is big, concluding the proof.
\end{itemize}

 
 
 
\subsection{Proof of \autoref{mainthm}} \label{sketchofproof} The rest of this section is dedicated to a detailed proof of \autoref{mainthm}. Note that if $X$ is $0$-dimensional, then the only ample divisor on $X$ is $\cO_X$ and the Theorem is trivially true. Hence, we assume for the rest of this section that $\dim X$ is positive.
 \begin{Not} 
 For any Cartier divisor $D$, we use the notation $H^{0}(D)$ to denote the space of global sections $H^{0}(X, \cO_{X}(D))$.
 \end{Not}
 
\begin{dfn} \label{Iedfn}
      For any Cartier divisor $D$ on $X$, define the $k$-vector subspace $I_{e}(D)$ of $H^{0}(D)$ as follows:
      $$ I_{e} (D) := \{ f \in H^{0}(D)\ | \ \varphi(F^e_*f )= 0 \text{ for all } \varphi \in \Hom_{\cO_X} (F^e_*\cO_X(D), \cO_X) \} .$$

      That is, $I_{e}(D)$ is the set of global sections $f$ of $\cO_{X}(D)$ such the map $\cO_{X} \to \Fe \cO_{X}(D)$ sending $1 \mapsto \Fe f$ does not split. A section $f \in H^{0}(D)$ that is not contained in $I_{e}(D)$ along with a map $\varphi: \Fe \cO_{X}(D)  \to  \cO_{X}$ sending $\Fe f$ to $1$ is called an \emph{$e^\text{th}$-Frobenius splitting} of the linear system $|D|$.
\end{dfn}
\medskip

\begin{lem} \label{formulaforFsig}
Let $L$ be an ample Cartier divisor and $S$ denote the section ring of $X$ with respect to $L$. Then, for any $e \geq 1$, if $a_{e}(L)$ denotes the free-rank of $\Fe S$ as an $S$-module, then $a_{e}(L)$ is computed by the following formula:
\begin{equation} \label{freerankformula}
    a_{e}(L) = \sum _{m = 0} ^{\infty} \dim_{k} \frac{H^{0}(mL)}{I_{e}(mL)}.
\end{equation}
Hence, the $F$-signature of $L$ can be computed as
$$ s_X(L) = \lim _{e \to \infty } \frac{  \sum \limits _{m = 0} ^{\infty} \dim_{k} \frac{H^{0}(mL)}{I_{e}(mL)}}{p^{e(\dim(X)+1)}}$$
\end{lem}

\begin{proof}
Let $\cL$ denote the invertible sheaf $\cO_{X}(L)$. We note that the $S$-module $\Fe S$ naturally decomposes as an $\NN$-graded module as (see the discussion preceding \autoref{Fsigofgradedrings}):
$$ \Fe S =  \bigoplus   _{n =0} ^{p^{e}-1} M_{e,n},      $$
where $M_{e,n} :=  \bigoplus _{i \geq 0} H^{0}(X, \cL^{i} \otimes \Fe\cL^{n})$ is naturally an $\NN$-graded $S$-module. Note also that since $H^{0}(X, \cL^{n+ip^e}) = 0$ for $i <0$ and $0 \leq n \leq p^e -1$, the module $M_{e,n}$ is the section module of the sheaf $\Fe \cL^n$ with respect to $\cL$.
We recall that by \autoref{fvsgf}, $a_{e}(L)$ can be calculated as the graded free-rank of $\Fe S$ i.e.
\begin{equation*} 
    \begin{split} 
    a_{e}(L) = \max \{r \, | \, & \Fe  S \isom \bigoplus _{t=1} ^{r} S(-j_{t}) \, \bigoplus \, N \, \text{as graded $S$-modules} \\
    & \text{for some $j_t \in \ZZ$  and some graded $S$-module $N$} \}     .
    \end{split} 
\end{equation*}

Since $\Fe S$ is $\NN$-graded, we note that each integer $j_{t}$ occurring in any decomposition of $\Fe S$ as above is non-negative. Sheaf theoretically, we have an equivalent description (see \cite[Theorem 3.10]{SmithGloballyFRegular} and the proof):
\begin{equation}   
\begin{split}
    a_{e}(L) =  \max \{r \, | \, &  \widetilde{\Fe S} \isom \bigoplus  _{0 \leq n \leq p^e -1} \Fe \cL^n   \isom \bigoplus _{t=1} ^{r} \cL^{-j_{t}} \bigoplus \,  \cN \\
    & \text{as $\cO_{X}$-modules for some $j_t \in \NN$  and some sheaf $\cN$} \} \end{split}   \end{equation}
 Now, for any $0 \leq n \leq p^{e}-1$, and $j \geq 0$, the maximum number of $\cL^{-j}$ summands of $\Fe \cL^{n}$ is the same as the maximum number of $\cO_{X}$-summands of $\Fe\cL^{n+jp^{e}}$. Writing $\Fe \cL^n \isom \cO_X ^{\oplus n} \oplus \cG$ such that $\cG$ does not have any $\cO_X$-summands, we see that the set $I_e (\cL^n)$ can be identified with the set $H^0 (X, \cG)$. Hence, the maximum number of $\cO_{X}$-summands of any $\Fe\cL^{m}$ is exactly given by the dimension of $H^{0}(mL)/I_{e}(mL)$ (see \autoref{lem.directsummand}, part (b)). Using \autoref{lem.directsummand} again, running over all $0 \leq n \leq p^{e}-1$ and $j \geq 0$, we get the desired formula~(\ref{freerankformula}) for $a_{e}(L)$. 
\end{proof}

\begin{rem}
Since the free-rank of $\Fe S$ is bounded by its generic rank (which is exactly $p^{e(\dim(X)+1)}$), the sum in equation~(\ref{freerankformula}) is indeed finite. Next, we will find uniform bounds for the number of terms in this sum.
\end{rem}

\begin{thm}\label{lem.Biglowerdim}
 Let $X$ be a globally $F$-regular projective variety over $k$. Fix a norm $\| \, \|$ on the N\'eron-Severi space $N^1 _\RR (X)$. There exists a constant $C_{1}:=C_{1}(X)$ (depending only on $X$, and the norm $\| \, \|$) such that, whenever $L$ and $H$ are any two effective Cartier divisors on $X$, we have:
\begin{enumerate}
    \item $$ I_{e}(mL) = H^{0}(mL) \textrm{ for $m > \frac{C_{1}}{\|L\|} p^{e}$}, \quad \text{and,}$$
    \item For all $n > \frac{2 \|H\|}{\|L\|}$, $$ I_{e} (m(nL+H)) = H^{0}(m(nL+H)) \textrm{ for all $m > \frac{C_{1}p^{e}}{n \|L\|}$}. $$
\end{enumerate}

\end{thm}

\begin{proof}
Since $X$ is normal, we can consider the canonical (Weil) divisor on $X$ (denoted by  $K_X$), by extending the canonical divisor on the non-singular locus of $X$. Choosing an ample divisor $A$ such that $A+ K_{X}$ is effective, we may write
$A \sim -K_{X} + E$ for some effective (Weil) divisor $E$. Let $[A]$ denote the class of $A$ in the ample cone of $X$. 

Let $L$ be any effective Cartier divisor on $X$. By applying duality for the Frobenius map, we have, 
$$\mathscr{H}om _{\cO_X}(F^e_*\cO_X(mL), \cO_X) \cong F^e_*\cO_X( -(p^e-1)K_X -mL).$$
See \cite[Section 4.1]{SchwedeSmithLogFanoVsGloballyFRegular} for a detailed discussion regarding duality for the Frobenius map. Hence, we have,
\begin{equation} \label{dualcriterion} \Hom _{\cO_X}(F^e_*\cO_X(mL), \cO_X) \isom \Fe H^{0}(X, \cO_{X}(-(p^e-1)K_{X} - mL)). \end{equation}
This shows that to prove that $H^{0}(mL) = I_{e}(mL)$ for any given $m$, it suffices to show the right hand side in \autoref{dualcriterion} is zero.

\vskip 8pt
{\it Claim: } There exists a positive constant $C_1 ^{\prime}$ (depending only on $X$, the choice of $A$ and the norm $\| \, \|$), such that for any effective divisor $D$ with $\|D\| > C_1 ^{\prime}$, we have $-K_X - D$ is not an effective divisor, i.e., $-K_X - D$ is not $\RR$-linearly equivalent to any effective divisor.

\vskip 8pt

\begin{proof}
[Proof of the claim]
Recall that the pseudoeffective cone (denoted by $\overline{\text{Eff}}(X)$) is a closed strongly convex cone (i.e. there is no non-zero class $\nu \in \overline{\text{Eff}}(X)$ such that $-\nu \in \overline{\text{Eff}}(X)$), and contains the class of every effective divisor on $X$ \cite[Definition 2.2.25]{LazarsfeldPositivity1}. Thus, the set 
$$ \kappa := \overline{\text{Eff}} (X) \bigcap ([A] -  \overline{\text{Eff}} (X))$$
is a compact subset of $\overline{\text{Eff}} (X)$. Since the norm function achieves a maximum on $\kappa$, we may choose $C_{1}^{\prime}$ to be bigger than the norm of any class in $\kappa$:
$$ C_{1} ^{\prime} > \max \{\|\xi\| \, | \, \xi \in \kappa \}.$$
Note that $C_{1}^{\prime}$ depends only on the choice of $A$ and the norm $\| \, \|$.

Since the class of every effective Cartier divisor on $X$ is contained in the pseudoeffective cone of $X$, if $D$ is an effective divisor with $\|D\| > C_{1} ^{\prime}$, then $D$ can not belong to $\kappa$ by the definition of $C_{1}^{\prime}$. Hence, we see that $A-D$ is not effective. Since $A = -K_X + E$ for an effective divisor $E$, this means that $-K_X - D$ is not effective. This proves the claim.
\end{proof}

{ \it Continuation of the proof of \autoref{lem.Biglowerdim}:} For any effective Cartier divisor $L$, if $m > \frac{C_{1}^{\prime} p^e}{\|L\| }$, we have $\|\frac{m}{p^e-1}L\| > C_{1} ^{\prime}$, hence, applying the claim above, we conclude that $-K_X - \frac{m}{p^e-1}L$ is not effective. Therefore, the divisor
$$ -(p^e-1) K_X - mL$$
 is not effective. By (\ref{dualcriterion}), this gives us $H^{0}(mL) = I_{e}(mL)$ as required. This proves part (a).

 For part (b), we use part (a) of the Theorem by replacing $L$ by $nL +H$, which gives us that $H^{0}(m(nL+H)) = I_{e}(m(nL+H))$ for $m> \frac{C_{1}^{\prime} p^e}{\|nL+H\|}$.
 Since by assumption $\| H \| \leq \frac{1}{2}\| nL\|$, we have, $\|nL +H \|\geq \|nL\| -\|H\| \geq \frac{1}{2}\|nL\|$. Therefore, $$\frac{2C_{1}^{\prime} p^e}{n\|L\|}\geq\frac{C_{1}^{\prime} p^e}{\| nL+H\|},$$ using which we see that $C_{1} = 2 C_{1}^{\prime}$ works for both parts (a) and (b). This completes the proof of \autoref{lem.Biglowerdim}.
\end{proof}

\begin{rem}
    For a more effective, but less uniform version of \autoref{lem.Biglowerdim}, see \autoref{lem.lowerdim}.
\end{rem}

Next, we prove \autoref{mainthm} in the special case when the divisor $L$ is fixed and the difference $L^{\prime} - L$ is big.

  \begin{lem}[Key Lemma] \label{ampledifference}
  Let $L$ be an integral ample divisor on $X$. Then, there exists a constant $C(L)$ (depending only on $L$ and the norm $\| \, \| $) such that for any other ample $\QQ$-divisor $L^{\prime}$ sufficiently close to $L$, and for which $L^{\prime}-L$ is big, we have:
  $$ |s_X(L) - s_X(L^{\prime})| \leq C(L) \|L - L^{\prime} \|.        $$
  \end{lem}

\begin{proof}[Proof of \autoref{ampledifference}] Throughout the proof, we fix the following set-up: Fixing the ample, integral divisor $L$ on $X$, we pick an arbitrary ample $\QQ$-divisor $L^{\prime}$ such that $L^{\prime} - L$ is big. Then, we may write $L^{\prime} = L + \frac{1}{n} H$, for some $n \gg 0$ and an effective and big Cartier divisor $H$.
\medskip

\begin{lem} \label{inclusion}
For effective divisors $D_1$ and $D_2$, consider the natural inclusion:
$$ \phi: \Fe \cO_X(D_1) \subset \Fe \cO_X(D_1 + D_2).     $$
Then, 
\begin{equation}
    \phi(I_{e}(D_1)) \subset I_{e}(D_1 + D_2).
\end{equation}
Equivalently, viewing $H^{0}(X, \cO_X(D_1))$ as a subset of $H^{0}(X, \cO_X(D_1 + D_2))$ through the map $\phi$, we have:
\begin{equation}
    \phi(I_{e}(D_1)) \subset I_{e}(D_1 + D_2) \cap H^{0}(D_1) = \{ x \in H^{0}(D_1) | \phi(x) \in I_{e}(D_{1} + D_{2}) \}.
\end{equation}
\end{lem}

\begin{proof}
This follows from the definitions once we observe that for every map $\varphi$ in $\Hom_{\cO_{X}}(\Fe \cO_{X}(D_1 + D_2), \cO_{X})$, we get a map $\Tilde{\varphi}$ in $\Hom_{\cO_{X}}(\Fe \cO_{X}(D_1), \cO_{X})$ by pre-composing with the map $ \phi$.
$$\begin{tikzcd}
    \Fe\cO_{X}(D_1) \arrow[dr, "\Tilde{\varphi}"] \arrow[r, "\phi"] &  \Fe\cO_{X}(D_1 + D_2) \arrow[d, "\varphi"] \\
    & \cO_{X}
\end{tikzcd}      $$
\end{proof}

\begin{lem} \label{factorization}
With $L$ an ample Cartier divisor and $H$ an effective big divisor on $X$,  and any natural number $n$, suppose that we have a natural number $b:=b(n)$, such that $nL-bH$ is big. Consequently, by \cite[Corollary 2.2.10]{LazarsfeldPositivity1}, there is a $C_{2} \gg 0$ such that for all $m \geq C_{2}$, we have
$$H^{0}(m(nL-bH)) \neq 0 .$$
Then, for $m \geq C_{2}$ and all $e \geq 1$, there is a factorization of inclusions:
$$
\begin{tikzcd}[row sep=large, column sep=large]
    \Fe\cO_{X}(mnbL) \arrow[dr, "\cdot \Fe c"] \arrow[r, " \cdot \Fe mb H"] &  \Fe\cO_{X}(mb(nL+H)) \arrow[d, "\cdot \Fe d"] \\
    & \Fe \cO_{X}(mn(b+1)L)
\end{tikzcd}
$$
given by a choice of a section $d \in H^{0}(X, \cO_X(mnL -mbH))$.
\end{lem}

\begin{proof}
Given a section $d \in H^{0}(X, \cO_X(mnL -mbH))$, let $D_{1}$ be the corresponding effective divisor. Then, we have $D_2 = mbH + D_{1} \sim mbH + mnL - mbH = mnL$. Then, we get inclusions
$$ \begin{tikzcd}[row sep=large, column sep=large]
    \cO_{X}(mnbL) \arrow[dr] \arrow[r] &  \cO_{X}(mb(nL+H)) \arrow[d] \\
    & \cO_{X}(mnbL + D_{2})
\end{tikzcd} $$
since $D_1$ was effective. We get the required factorization by applying $\Fe$ to the above diagram and taking $c$ to be the section corresponding to the divisor $D_{2}$.
\end{proof}

\begin{lem} \label{maininequalitylem}
Let $d=\dim X$ and let $C_1 = C_1(X)$ be the constant as obtained in \autoref{lem.Biglowerdim}. Fix an ample Cartier divisor $L$ and an effective Cartier divisor $H$ on $X$. Let $n$ and $b:=b(n)$ be positive integers such that  $n > \frac{2 \|H\|}{\|L\|}$ and that $nL-bH$ is big. Then, we have the following inequality:


\begin{equation} \label{maininequality}
    \begin{split}
     |s_X(L) - s_X(L + \frac{1}{n} H)|  
       \leq & \quad  \frac{C_{1}^{d+1}}{\|L\|^{d+1} (d+1)!} \, \Bigg(  2  \, \vol(L) \, \frac{\big( (b+1)^{d} - b^{d} \big )}{b^{d}} \, \\
       & + \, \big( \vol(L + \frac{1}{n}H) -\vol(L) \big) \Bigg) + 2 \,  \frac{s_X(L)}{b+1} 
    \end{split}
\end{equation}
\end{lem}


\begin{proof}

First, fixing $n$ and $b$, there is a $C_{2} \gg 0$ such that $H^{0}(m(nL-bH)) \neq 0$ for all $m \geq C_{2}$. Using \autoref{formulaforFsig} and \autoref{lem.Biglowerdim}, we have the following formulas for the $F$-signatures:
\begin{equation} \label{formulaforFsig1} s_X(nbL) = \lim_{e \to \infty} \frac{1}{p^{e(d+1)}}\sum \limits _{m = C_{2}} ^{\frac{C_{1}p^e}{\|L\|nb}}\dim_{k}\frac{  H^{0}(mnbL)}{I_{e}(mnbL)} \end{equation}
and similarly,
\begin{equation} \label{formulaforFsig2} s_X(b(nL+H)) =
\lim_{e \to \infty} \frac{1}{p^{e(d+1)}}\sum \limits _{m = C_{2}} ^{\frac{C_{1}p^e}{\|L\|nb}}\dim_{k}\frac{  H^{0}(mb(nL+H))}{I_{e}(mb(nL+H))}.\end{equation}

Note that even though the formula from \autoref{formulaforFsig} requires us to begin the sums (\ref{formulaforFsig1}) and (\ref{formulaforFsig2}) at $m =0$, we may begin the sums at $C_{2}$ since changing finitely many terms does not alter the limit.

According to formulas (\ref{formulaforFsig1}) and (\ref{formulaforFsig2}), to compare $s_X(nbL)$ with $s_X(b(nL +H))$, we need to understand the difference
$$ \dim_k \frac{H^0(mbnL)}{I_e(mbnL)} - \dim_k \frac{H^0 (mb(nL+H))}{I_e(mb(nL+H))} .$$


We have an inclusion
\begin{equation} \label{comparisonmaps}
\frac{H^0(mbnL)}{H^0(mbnL)\cap I_e(mb(nL+H))} \hookrightarrow \frac{H^0 (mb(nL+H))}{I_e(mb(nL+H))} \end{equation}
coming from the inclusion of $H^{0}(mbnL) \hookrightarrow H^{0}(mb(nL+H))$.
\medskip

Let $J_e(mbnL) = H^0(mbnL)\cap I_e(mb(nL+H))$. Then using (\ref{comparisonmaps}), we have:
\begin{equation} \label{dimensioncomparison}
    \dim_{k} \frac{H^0(mb(nL+H))}{I_e(mb(nL+H))} = \dim_k \frac{H^0(mbnL)}{J_e(mbnL)} + \dim_{k} \frac{H^0(mb(nL+H))}{H^{0}(mnbL) + I_{e}(mb(nL+H))}.
\end{equation}

Then, using (\ref{dimensioncomparison}) and the triangle inequality, we get

\begin{equation} \label{intermediateinequlaity}
    \begin{aligned}
    &   \left| \sum_{m=C_{2}}^{\frac{C_{1}p^e}{\|L\|nb}}\dim_k\frac{H^0(mbnL)}{I_e(mbnL)} - \sum_{m=C_{2}}^{\frac{C_{1}p^e}{\|L\|nb}}\dim_k \frac{H^0(mb(nL+H))}{I_e(mb(nL+H))} \right| \\
    &   \left| \sum_{m=C_{2}}^{\frac{C_{1}p^e}{\|L\|nb}}\dim_k\frac{H^0(mbnL)}{I_e(mbnL)} - \sum_{m=C_{2}}^{\frac{C_{1}p^e}{\|L\|nb}} \Big( \dim_{k} \frac{H^0(mbnL)}{J_e(mbnL)} + \dim_{k} \frac{H^0(mb(nL+H))}{H^{0}(mnbL) + I_{e}(mb(nL+H))} \Big) \right|  \\
    \leq &    \left| \sum_{m=C_{2}}^{\frac{C_{1}p^e}{\|L\|nb}}\dim_k\frac{H^0(mbnL)}{I_e(mbnL)} -\sum_{m=C_{2}}^{\frac{C_{1}p^e}{\|L\|nb}}\dim_k \frac{H^0(mbnL)}{J_e(mbnL)}  \right| 
    +  \sum_{m=C_{2}}^{\frac{C_{1}p^e}{\|L\|nb}}\dim_{k} \frac{H^0(mb(nL+H))}{H^{0}(mnbL)}  \\
    \leq &    \left|\sum_{m=C_{2}}^{\frac{C_{1}p^e}{\|L\|nb}} \dim_k\frac{H^0(mbnL)}{I_e(mbnL)} -\sum_{m=C_{2}}^{\frac{C_{1}p^e}{\|L\|nb}}\dim_k \frac{H^0(m(b+1)nL)}{I_e(m(b+1)nL)}  \right| \\
    + &    \left|\sum_{m=C_{2}}^{\frac{C_{1}p^e}{\|L\|nb}} \dim_k\frac{H^0(m(b+1)nL)}{I_e(m(b+1)nL)} -\sum_{m=C_{2}}^{\frac{C_{1}p^e}{\|L\|nb}}\dim_k \frac{H^0(mbnL)}{J_e(mbnL)}  \right| 
     +   \sum_{m=C_{2}}^{\frac{C_{1}p^e}{\|L\|nb}}\dim_{k} \frac{H^0(mb(nL+H))}{H^{0}(mnbL)} 
   \end{aligned}
\end{equation}
where in the last inequality, we use the triangle inequality again after adding and subtracting the term $\sum \limits _{m=C_{2}}^{\frac{C_{1}p^e}{\|L\|nb}} \dim_k\frac{H^0(m(b+1)nL)}{I_e(m(b+1)nL)}$.

\medskip

To proceed, we need to understand the difference between the spaces $\frac{H^0(m(b+1)nL)}{I_e(m(b+1)nL)}$ and  $\frac{H^0(mbnL)}{J_e(mbnL)}$. To this end, we prove the following:

\begin{lem} \label{colon}
Suppose, as in \autoref{factorization}, $b$ is such that $nL-bH$ is big and $C_{2}$ is such that for all $m \geq C_{2}$, we have $H^{0}(m(nL-bH)) \neq 0 $. Then, for $m \geq C_{2}$ and all $e \geq 1$, choosing a non-zero global section $d \in H^{0}(mnL -mbH)$ and setting $c = d \otimes h^{mb}$, where $h$ is the section of $\cO_{X}(H)$ that corresponds to the rational function $1$, we have the inclusions
\begin{equation} \label{inclusions}  I_{e}(mnbL) \subset J_{e}(mnbL) \subset \{ x \in H^{0}(mnbL) \, | \, cx \in I_{e}(mn(b+1)L)\}.    \end{equation}
Moreover, we have the following inequality (with $C_{1}$ being the constant from \autoref{lem.Biglowerdim}):

\begin{equation} \label{secondmaininequality}
    \begin{aligned}
       & \left|  \sum_{m=C_{2}}^{\frac{C_{1}p^e}{\|L\|nb}} \dim_k \frac{ H^{0}( mn(b+1)L)}{I_{e}(mn(b+1)L)} - \sum_{m=C_{2}}^{\frac{C_{1}p^e}{\|L\|nb}} \dim _k \frac{H^{0}(mnbL)}{J_e(mnbL)}\right| \\
       \leq &     \sum_{m=C_{2}}^{\frac{C_{1}p^e}{\|L\|nb}}2\dim_{k} \frac{ H^{0}(mn(b+1)L)}{c H^{0}(mnbL)} +        \left| \sum_{m=C_{2}}^{\frac{C_{1}p^e}{\|L\|nb}}\dim_k \frac{ H^{0}(mnbL)}{I_e(mnbL)} -\sum_{m=C_{2}}^{\frac{C_{1}p^e}{\|L\|nb}} \dim_{k} \frac{ H^{0}(mn(b+1)L)}{I_e(mn(b+1)L)}\right|
    \end{aligned}
\end{equation}
\end{lem}

Before proving \autoref{colon}, we note that putting (\ref{secondmaininequality}) together with (\ref{intermediateinequlaity}), we obtain:

\begin{equation*}
    \begin{aligned}
    &   \left| \sum_{m=C_{2}}^{\frac{C_{1}p^e}{\|L\|nb}}\dim_k\frac{H^0(mbnL)}{I_e(mbnL)} - \sum_{m=C_{2}}^{\frac{C_{1}p^e}{\|L\|nb}}\dim_k \frac{H^0(mb(nL+H))}{I_e(mb(nL+H))} \right| \\
    \leq &   \sum_{m=C_{2}}^{\frac{C_{1}p^e}{\|L\| nb}} 2 \dim_{k} \frac{ H^{0}(mn(b+1)L)}{ H^{0}(mnbL)} +  \sum_{m=C_{2}}^{\frac{C_{1}p^e}{\|L\|nb}} \dim_{k} \frac{H^0(mb(nL+H))}{H^{0}(mnbL)}   \\
    + & 2\left| \sum_{m=C_{2}}^{\frac{C_{1}p^e}{\|L\|nb}}\dim_k \frac{ H^{0}(mnbL)}{I_e(mnbL)} -\sum_{m=C_{2}}^{\frac{C_{1}p^e}{\|L\|nb}} \dim_{k} \frac{ H^{0}(mn(b+1)L)}{I_e(mn(b+1)L)}\right|
   \end{aligned}
\end{equation*}
\medskip

Hence, we get
\begin{equation*}
    \begin{aligned}
    & \left| s_X(nbL) - s_X(b(nL+H))\right|
    = \lim _{e \to \infty} \frac{1}{p^{e(d+1)}}  \left|  \sum  _{m = C_{2}} ^{\frac{C_{1}p^{e}}{\|L\|nb}}  \dim_k\frac{H^0(mbnL)}{I_e(mbnL)} -   \sum  _{m = C_{2}} ^{\frac{C_{1}p^{e}}{\|L\|nb}} \dim_k \frac{H^0(mb(nL+H))}{I_e(mb(nL+H))} \right|  \\
   & \leq   \lim_{e \to \infty} \frac{1}{p^{e(d+1)}} \Bigg( \sum  _{m = C_{2}} ^{\frac{C_{1}p^{e}}{\|L\|nb}}  2 \dim_{k} \frac{ H^{0}(mn(b+1)L)}{ H^{0}(mnbL)} +     \sum  _{m = C_{2}} ^{\frac{C_{1}p^{e}}{\|L\|nb}}  \dim_{k} \frac{H^0(mb(nL+H))}{H^{0}(mnbL)}   \\
    &    +      2\left|\sum  _{m = C_{2}} ^{\frac{C_{1}p^{e}}{\|L\|nb}} \dim_k \frac{ H^{0}(mnbL)}{I_e(mnbL)} -  \sum  _{m = C_{2}} ^{\frac{C_{1}p^{e}}{\|L\|nb}}\dim_{k} \frac{ H^{0}(mn(b+1)L)}{I_e(mn(b+1)L)}\right| \Bigg)\\
    &   \leq \lim_{e \to \infty}  \frac{C_{1}^{d+1}  p^{e(d+1)}}{\|L\|^{d+1} n^{d+1}b^{d+1}p^{e(d+1)}(d+1)!} \, \Bigg(  2 n^{d} \, \vol(L) \, \big( (b+1)^{d} - b^{d} \big ) \,  + \,  b^{d}n^{d} \, \big( \vol(L + \frac{1}{n}H) -\vol(L) \big) \Bigg) \\
    &   + 2 \, \left| s_X(nbL) - s_X(n(b+1)L) \right| \\
    &   = \frac{C_{1}^{d+1} n^{d} b^{d}}{\|L\|^{d+1} n^{d+1} b^{d+1} (d+1)!} \Bigg(  2 \vol(L)\big(\frac{(b+1)^d - b^{d}}{b^{d}} \big) + \big( \vol(L+\frac{1}{n}H) - \vol(L) \big) \Bigg) \\
    &   + \left| s_X(nbL) - s_X(n(b+1)L) \right|
   \end{aligned}
\end{equation*}

Finally, using the scaling property for $s_X$ (\autoref{propertiesofs}), we get:
\begin{equation*}
    \begin{aligned}
    & \left| s_X(L) - s_X(L+\frac{1}{n} H) \right| \\
    &   = nb \, \left| s_X(nbL) - s_X(b(nL+H)) \right| \\
    &   \leq \frac{C_{1}^{d+1}}{\|L\|^{d+1} (d+1)!} \, \Bigg(  2  \, \vol(L) \, \frac{\big( (b+1)^{d} - b^{d} \big )}{b^{d}} \,  + \, \big( \vol(L + \frac{1}{n}H) -\vol(L) \big) \Bigg) \\
    &   + 2 \, nb \, \left| \frac{s_X(L)}{nb} - \frac{s_X(L)}{n(b+1)} \right| \\
    &   = \frac{C_{1}^{d+1}}{\|L\|^{d+1} (d+1)!} \, \Bigg(  2  \, \vol(L) \, \frac{\big( (b+1)^{d} - b^{d} \big )}{b^{d}} \,  + \, \big( \vol(L + \frac{1}{n}H) -\vol(L) \big) \Bigg)   + 2 \,  \frac{s_X(L)}{b+1} 
   \end{aligned}
\end{equation*}
This completes the proof of \autoref{maininequalitylem}, pending the proof of \autoref{colon}, which we prove next.
\end{proof}

\begin{Not}
Recall that $L$ and $H$ are fixed integral Cartier divisors, with $L$ ample and $H$ effective and $n \geq 1$ is any natural number. For any natural number $k \in \NN$, we define:
$$ I_{e}(k) := I_{e}(kL),$$
$$ J_{e}(kn) := H^{0}(knL) \cap I_{e}(k(nL + H)), $$
where we view $H^{0}(knL)$ as a subspace of $H^{0}(k(nL + H))$ via the inclusion map $\cO_{X}(nkL) \subset \cO_{X}(knL + kH) $.
\end{Not}

\begin{proof}[Proof of \autoref{colon}]
The first inclusion in (\ref{colon}) follows from \autoref{inclusion} by taking $D_1 = mnbL$ and $D_2 = mnbH $. The second inclusion follows from \autoref{factorization} and the second part of \autoref{inclusion}, by taking $D_1 = mb(nL+H)$ and $D_2$ to be the effective divisor corresponding to $ d \in H^{0}(mnL -mbH)$.
  Hence, we get
\begin{equation*}
    \begin{aligned}
       & \left| \sum_{m=C_{2}}^{\frac{C_{1}p^e}{\|L\|nb}} \dim_k \frac{ H^{0}( mn(b+1)L)}{I_{e}(mn(b+1))} - \sum_{m=C_{2}}^{\frac{C_{1}p^e}{\|L\|nb}} \dim _k \frac{H^{0}(mnbL)}{J_e(mnb)}\right|  & \\ 
     = &  \left| \sum_{m=C_{2}}^{\frac{C_{1}p^e}{\|L\|nb}}\dim_k \frac{H^{0}(mn(b+1)L)}{c H^{0}(mnbL)} -\sum_{m=C_{2}}^{\frac{C_{1}p^e}{\|L\|nb}} \dim_{k} \frac{I_e(mn(b+1))}{cJ_e(mnb)}\right| \quad \textrm{(rearranging terms)}\\
        \leq  &  \sum_{m=C_{2}}^{\frac{C_{1}p^e}{\|L\|nb}} \dim_k \frac{H^{0}(mn(b+1)L)}{c H^{0}(mnbL)} + \sum_{m=C_{2}}^{\frac{C_{1}p^e}{\|L\|nb}}\dim_{k} \frac{I_e(mn(b+1))}{cJ_e(mnb)} \quad \textrm{(triangle inequality)} \\
        \leq  &  \sum_{m=C_{2}}^{\frac{C_{1}p^e}{\|L\|nb}} \dim_k \frac{H^{0}(mn(b+1)L)}{c H^{0}(mnbL)} +\sum_{m=C_{2}}^{\frac{C_{1}p^e}{\|L\|nb}} \dim_{k} \frac{I_e(mn(b+1))}{cI_e(mnb)} \quad \textrm{(since $cI_{e}(mnb) \subset cJ_{e}(mnb)$ by (\ref{inclusions}))} 
    \end{aligned}
     \end{equation*}
    \begin{equation*}
        \begin{aligned}
        =  & \sum_{m=C_{2}}^{\frac{C_{1}p^e}{\|L\|nb}} \dim_k \frac{H^{0}(mn(b+1)L)}{c H^{0}(mnbL)} + \left|\sum_{m=C_{2}}^{\frac{C_{1}p^e}{\|L\|nb}} \dim_{k} \frac{ H^{0}(mn(b+1)L)}{cI_e(mnb)} - \sum_{m=C_{2}}^{\frac{C_{1}p^e}{\|L\|nb}}\dim_{k} \frac{ H^{0}(mn(b+1)L)}{I_e(mn(b+1))} \right|  \\
         = & \sum_{m=C_{2}}^{\frac{C_{1}p^e}{\|L\|nb}} \dim_k \frac{H^{0}(mn(b+1)L)}{c H^{0}(mnbL)} \\
          & + \left|\sum_{m=C_{2}}^{\frac{C_{1}p^e}{\|L\|nb}} \left(\dim_{k} \frac{ H^{0}(mn(b+1)L)}{c H^{0}(mnbL)} 
         + \dim_k \frac{ H^{0}(mnbL)}{I_e(mnb)} - \dim_{k} \frac{ H^{0}(mn(b+1)L)}{I_e(mn(b+1))}\right)\right| \\
          \leq  & \sum_{m=C_{2}}^{\frac{C_{1}p^e}{\|L\|nb}}   2 \dim_{k} \frac{ H^{0}(mn(b+1)L)}{c H^{0}(mnbL)} +        \left|\sum_{m=C_{2}}^{\frac{C_{1}p^e}{\|L\|nb}} \dim_k \frac{ H^{0}(mnbL)}{I_e(mnb)} - \sum_{m=C_{2}}^{\frac{C_{1}p^e}{\|L\|nb}} \dim_{k} \frac{ H^{0}(mn(b+1)L)}{I_e(mn(b+1))}\right|
    \end{aligned}
\end{equation*}
where in the second-last step, we rearrange the terms of the sum, and in the last step use the triangle inequality again.
This completes the proof of the lemma.
\end{proof}
\medskip

To complete the proof of \autoref{ampledifference}, we need the following lemma about the Lipschitz continuity of the volume function. We record a quick proof for ample classes that works for any algebraically closed field, and in any characteristic:
 
\begin{lem} \label{volumeLipscitz}\cite[Theorem 2.2.44]{LazarsfeldPositivity1}
    Let $X$ be a projective variety of dimension $d$ over $k$. Fix a norm $\| \, \|$ on the real N\'eron-Severi space. Then, there exists a positive constant $C>0$ such that for any two real \emph{ample} classes $\xi$ and $\xi^\prime$, we have:
    $$ | \vol(\xi) - \vol(\xi^\prime) | \leq C \max(\|\xi\|, \| \xi^\prime\|) ^{d-1} \| \xi - \xi^\prime \| . $$ 
\end{lem}

\begin{proof}

Since the volume function coincides with the intersection form on the real Nef cone, it is given by a polynomial $P$ of degree $d$ once we choose a basis for $N^{1}_{\RR}(X)$. Hence, there exists a constant $C$ (depending only on $X$), such that $$\|P^{\prime} (x_{1}, \dots, x_{\rho}) \| \leq  C \|(x_{1}, \dots, x_{\rho}) \|^{d-1}$$
for any vector $(x_{1}, \dots, x_{\rho}) \in \Nef_{\RR}(X) $. With this observation, the Lemma follows from an application of the mean-value theorem.

\end{proof}

\paragraph{\textit{Completion of the proof of \autoref{ampledifference}}:} Recall that $L$ is a fixed ample divisor on $X$ (in particular, $L$ is big). Suppose $L^{\prime}$ is an ample $\QQ$-divisor such that $L^{\prime}-L$ is big. Further assume that $ \| L^\prime - L \| < \frac{\|L\|}{2}$. Then, we may write $L^{\prime} = L + \frac{1}{n}H$ for a suitable effective Cartier divisor $H$ and some natural number $n \geq 1$. 

We would like to apply \autoref{maininequalitylem} to this choice of $L$, $H$ and $n$. For this, we need to choose a natural number $b$ such that $nL - bH$ is big. We note that we may choose $b$ in the following way:
 Since $L$ is big, by openness of the big cone of $X$, there exists a constant $C_4 > 0$ (depending only on $L$) such that any $\QQ$-divisor $D$ satisfying $\| L-D \| \leq C_4$ is also big. Since we need $L -\frac{b}{n}H$ to be big, it is sufficient that $\| \frac{b}{n} H \| \leq C_4$. So we may choose $b(n) = \lfloor \frac{n C_4}{\| H \|} \rfloor$ so that $b(n) \to \infty$ as $n \to \infty$.
 
 Now, applying \autoref{maininequalitylem} to this choice of $n$ and $b$, we get:
 \begin{equation} \label{maininequality2}
  \begin{split}
     |s_X(L) - s_X(L + \frac{1}{n} H)|  
       \leq & \quad  \frac{C_{1}^{d+1}}{\|L\|^{d+1}(d+1)!} \, \Bigg(  2  \, \vol(L) \, \frac{\big( (b+1)^{d} - b^{d} \big )}{b^{d}} \, \\
       & + \, \big( \vol(L + \frac{1}{n}H) -\vol(L) \big) \Bigg) + 2 \,  \frac{s_X(L)}{b+1} 
    \end{split}        \end{equation}
 
 Further, we have
 $$ \frac{(b+1)^{d} -b^{d}}{b^{d}} \leq \frac{2^d}{b}    $$
 and by \autoref{volumeLipscitz}, there is a positive constant $C_{3}$, depending only on $X$ and the norm $\| \, \|$, such that for any two ample classes $\xi_{1}, \xi_{2} \in N^{1} _{\QQ}(X)$,
 
 $$ | \vol(\xi_{1}) - \vol(\xi_{2})| \leq C_{3} \, \big(\max(\|\xi_{1}\|, \| \xi_{2} \|) \big)^{d-1} \|\xi_{1} - \xi_{2} \|  .  $$
 Putting these together, along with (\ref{maininequality2}), and using that $\| \frac{1}{n}H \| = \| L - L^\prime \|$, we get
\begin{equation}
    \begin{split}
     |s_X(L) - s_X(L^\prime)|  
       \leq & \quad  \frac{C_{1}^{d+1}}{(d+1)!} \, \Bigg(  2  \, \vol(L) \, \frac{2^{d}}{b} \, \\
       & + \,  C_{3} \, \|L^{\prime}\|^{d-1} \|L^{\prime} - L\|  \Bigg) +  \frac{2}{b}s_X(L) 
    \end{split}  
\end{equation}

Next, using the fact $b$ was chosen to be $b(n) = \lfloor \frac{n C_4}{\| H \|} \rfloor$, we have $b \geq \frac{n C_4}{2\|H\|}$, using which we get
\begin{equation}
    \begin{split}
     |s_X(L) - s_X(L^\prime)|  
       \leq & \quad  \frac{C_{1}^{d+1}}{\|L^{d+1}\| (d+1)!} \, \Bigg(  2  \, \vol(L) \, \frac{2^{d+1}}{C_4} \|L - L^{\prime} \| \, \\
       & + \,  C_{3} \, \|L^{\prime}\|^{d-1} \|L - L^{\prime} \|  \Bigg) +  \frac{4}{C_4} s_X(L)\|L - L^{\prime} \| 
    \end{split}  
\end{equation}

Lastly, since $\| L - L^\prime \| < \frac{\|L\|}{2}$ we have $\| L^\prime \| < 2 \| L \|$. Hence, we have
\begin{equation}
    \begin{split}
     |s_X(L) - s_X(L^\prime)|  
       \leq & \quad  \frac{C_{1}^{d+1}}{\|L\|^{d+1} (d+1)!} \, \Bigg(  2  \, \vol(L) \, \frac{2^{d+1}}{C_4} \|L - L^{\prime} \| \, \\
       & + \,  2^{d-1} C_{3} \, \|L\|^{d-1} \|L - L^{\prime} \|  \Bigg) +  \frac{4}{C_4} s_X(L)\|L - L^{\prime} \| 
    \end{split}  
\end{equation}
Hence, we see that for any ample, integral divisor $L$, we have
$$ | s_X(L) - s_X(L^\prime)| \leq C(L) \| L - L^\prime \|$$
for all ample $\QQ$-divisors $L^\prime$ such that $L^\prime - L$ is big and $\|L - L^\prime \| <  \frac{\|L\|}{2}$
where $C(L)$ is given by
$$ C(L) =  \frac{C_{1}^{d+1}}{\|L\|^{d+1} (d+1)!} \Big(  \vol(L)\frac{2^{d+1}}{C_4} + 2^{d-1} C_{3} \|L\|^{d-1} \Big) + \frac{4}{C_{4}}s_X(L) . $$
This completes the proof of the Key \autoref{ampledifference}.
\end{proof}

The proof of the Key \autoref{ampledifference} actually shows a stronger and more explicit statement that will be useful to us. We record it in the following Proposition.

\begin{Pn} \label{explicitLipschitz}
For any ample, integral divisor $L$, we have
$$ | s_X(L) - s_X(L^\prime)| \leq C(L) \| L - L^\prime \|$$
for all ample $\QQ$-divisors $L^\prime$ such that $L^\prime - L$ is big and $\|L - L^\prime \| <  \frac{\|L\|}{2}$,
where $C(L)$ maybe chosen to be of the form
$$ C(L) =  \frac{C_{1}^{d+1}}{\|L\|^{d+1} (d+1)!} \Big(  \vol(L)\frac{2^{d+1}}{C_4} + 2^{d-1} C_{3} \|L\|^{d-1} \Big) + \frac{4}{C_{4}}s_X(L) . $$ Here, $C_{1}:= C_1(X)$ is the constant (depending only on $X$) obtained in \autoref{lem.Biglowerdim}, $C_{3}$ depends only on $X$, and $C_{4}:= C_{4}(L)$ is any constant (depending on $L$) with the property that the closed ball $B = \{ D \in \text{N}^1 _\QQ (X) \, | \, \|D - L\| \leq C_{4} \}$ is contained in the big cone of $X$.
\end{Pn}

Next, we examine how the constant $C(L)$ in \autoref{explicitLipschitz} varies with $L$.

\begin{lem} \label{uniformconstants}
     Let $X$ be projective variety and $\cC$ be a closed cone contained in the big cone of $X$. Then, there exists a constant $\tilde{C_4}$ (depending only on $\cC$) such that for any non-zero class $D  \in \cC$, the closed ball
    $$ B (D) = \{\xi \in \mathrm{N}^{1}_{\RR} (X) \, | \, \| \xi - D \| < C_{4} \, \| D \| \} $$
    is contained in $\mathrm{Big}(X)$.
\end{lem}

\begin{proof}
Consider the set
        $$  \kappa := \{ D^{\prime} \in \cC \, | \, \|D^{\prime}\| = 1 \}.     $$
    Since $\cC$ is a closed cone, $\kappa$ is a compact subset of $\cC$. Moreover, since $\cC$ is contained in the big cone of $X$ and because $\mathrm{Big}(X)$ is an open subset of $\text{N}^{1} _{\RR}(X)$, there exists a positive real number $\tilde{C}_4 > 0$ such that the ball $B_{\tilde{C_4}} (D) = \{ \xi \, | \, \|D- \xi \| \leq \tilde{C}_4 \}$ is contained in the big cone for all $D \in \kappa$. Now the lemma follows by considering $\frac{1}{\|D\|} D \in \kappa $ whenever $D$ is a non-zero class in $\cC$.
\end{proof}

 \begin{lem} \label{normchange}
  Given any two norms $\| \   \|_{1}$ and $\| \ \|_{2}$ on the vector space $\RR^{N}$, we have positive constants $\mu_{1}$ and $\mu_{2}$ such that for any vector $\nu \in \RR^{N}$,
  $$ \mu_{1} \| \nu \|_{1} \leq \| \nu \|_{2} \leq \mu_{2} \|\nu \|_{1}. $$
  \end{lem}
  \begin{proof}
 See \cite[Section 5.1, Ex. 6]{FollandRealAnalysis}.\end{proof}

\begin{lem} \label{choosingr}
Let $e_1, \dots, e_\rho $ be a basis for the N\'eron-Severi space of $X$, where each $e_i$ corresponds to a big divisor. Let $\cC$ denote the closed cone generated by the $e_{i} $'s and $\| \, \|$ denote the sup-norm with respect to the basis $\{e_i\}$. For any $L$ in $\cC$, let $\lambda_{i}(L)$ denote the $i^{\text{th}}$-coordinate of $L$ with respect to the basis $\{e_i\}$. Suppose we have two positive numbers $ 0 < A_1 < A_2$ and a compact subset $\kappa$ of $\cC$ defined by
$$\kappa = \{ \xi \in \cC \, | \, A_1 \leq \|\xi \| \leq A_2 \}.$$
In this situation, for every $D$ in the interior of $\kappa$, there exists a positive real number $r(D)$ such that the following three conditions are satisfied:
 \begin{enumerate}
    \item $ r(D) < \frac{A_{1}}{2}$.
      \item The closed ball
  $$ B_{r(D)} := \{D^{\prime} \, | \, \| D^{\prime} - D\| \leq r(D) \} $$ 
  is contained in the interior of $\kappa$.
  
  \item For any two $\QQ$-divisors $L$ and $L^{\prime}$ in $ B_{r(D)}$, setting $\lambda = \max_{i} \{\frac{\lambda_{i}(L)}{\lambda_{i}(L^{\prime})}\}$, we have
  $$  A_{1} < \lambda \|L^{\prime} \|  < A_{2} $$
and $$ \|\lambda L^{\prime} - L \| < \frac{A_{1}}{2} .$$
  \end{enumerate}

\end{lem}
\begin{proof}
 First, pick any positive number $r < \frac{A_{1}}{4}$ such that $B_{r}$, the closed ball of radius $r$ around $D$ is contained in $\kappa$ (this is possible since $D$ is contained in the interior of $\kappa$). Now, there exists a positive number $\varepsilon$ such that for any $L$ in $B_{r/2}$, both $(1-\varepsilon)L$ and $(1+\varepsilon)L$ are contained in $B_{r}$. Finally pick $0 < r(D) < r/2$ so small that for each $i$, we have
    $$ \left| 1 - \frac{\lambda_{i}(L)}{\lambda_{i}(L^{\prime})} \right|< \varepsilon  $$
    for all $L, L^{\prime}$ in  $B_{r(D)}$. This is possible due to the local uniform continuity of the function $\lambda_{i}(L)$ as $L$ varies. By construction, for any $L, L^{\prime} \in B_{r(D)}$, we have 
    $$ 1 - \varepsilon < \lambda = \max_{i} \left\{\frac{\lambda_{i}(L)}{\lambda_{i}(L^{\prime})}\right\} < 1+ \varepsilon. $$
    This ensures that $\lambda L^{\prime}$ is in $B_{r}$ and since $r < \frac{A_{1}}{4}$, also that
    $$ \|\lambda L^{\prime} - L \| \leq \| \lambda L^{\prime} - D \| + \| D- L \| < \frac{3r}{2} <\frac{A_{1}}{2} .$$
\end{proof}

Finally, we can now prove \autoref{mainthm}.

\paragraph{\textit{Completion of the proof of \autoref{mainthm}:}} \label{finalstep} 
 Fix a real class $D$ in the ample cone $\Amp_{\RR}(X)$. Then, to prove that $s_X$ is locally Lipschitz around $D$, by \autoref{normchange}, we may a pick a suitable norm depending on $D$. Since the ample cone $\Amp_\RR (X)$ is an open subset of $N^1 _\RR(X)$, given $D$ in $\Amp_ \RR(X)$, we may pick a basis $e_{1}, \dots, e_\rho$ for $N^1 _\RR  (X)$ such that each $e_i$ is the class of an ample invertible sheaf and such that $D$ in contained in the interior of the cone generated by the $e_i$'s (equivalently, $D = \sum a_{i} e_i$ with each $a_i > 0$). Let $\cC = \{ a_{i} e_i \, | \, a_i \geq 0 \}$ denote the closed cone generated by the $e_i$'s and $\| \, \|$ denote the sup-norm with respect to the basis $\{e_i\}$.

Pick two positive real numbers $A_{1}$ and $A_2$ such that $0 < A_{1} < \|D \| < A_2 $. Let $\kappa = \{ D^\prime \in \cC  \, | \,  A_1 \leq \| D^\prime \| \leq A_2 \}$.  We will first consider the case of any two $\QQ$-divisors $L$ and $L^\prime$ in $\kappa$ such that $ L^{\prime} - L$ is big and $\| L^\prime - L \| <\frac{\|L\|}{2}$. Choose an integer $r \gg0$ such that $rL$ is integral. Then, we may apply \autoref{explicitLipschitz} to $rL$ and $rL^\prime$, to get

\begin{equation} \label{almostfinalequation} | s_X(rL) - s_X(rL^\prime)| \leq C(rL) \| rL - rL^\prime \| \end{equation}

where 
\begin{equation} \label{scalingofC} C(rL) =  \frac{C_{1}^{d+1}}{\|rL\|^{d+1} (d+1)!} \Big(  \vol(rL)\frac{2^{d+1}}{C_4(rL)} + 2^{d-1} C_{3} \|rL\|^{d-1} \Big) + \frac{4}{C_{4}(rL)}s_X(rL) . \end{equation}

Now, applying \autoref{uniformconstants} to the cone $\cC$, we may pick $C_{4}(rL)$ with the property that
$$  C_{4}(rL) \geq \tilde{C_{4}} \, \|rL\| $$
for some constant $ \tilde{C_4}$ (depending only on the basis $\{e_i\}$) and for all $r$ and all $L \in \cC$. Using this in (\ref{scalingofC}), we get

 \begin{equation} \label{finalsplitinequality}
 \begin{split}
  C(rL) & \leq  \frac{C_{1}^{d+1}}{\|rL\|^{d+1} (d+1)!} \Big(  \vol(rL)\frac{2^{d+1}}{\tilde{C_4}\|rL\|} + 2^{d-1} C_{3} \|rL\|^{d-1} \Big) + \frac{4}{\tilde{C_{4}}\|rL\|}s_X(rL) \\
  & =  \frac{C_{1}^{d+1}}{(d+1)!\|L\|^{d+1} r^{d+1}} \Big(  r^{d} \vol(L)\frac{2^{d+1}}{r \tilde{C_4}\|L\|} + 2^{d-1} C_{3} r^{d-1} \|L\|^{d-1} \Big) + \frac{4}{r^2 \tilde{C_{4}}\|L\|}s_X(L) \\
  & = \frac{1}{r^2} \Bigg(\frac{C_{1}^{d+1}}{(d+1)!\|L\|^{d+1} } \Big(   \vol(L)\frac{2^{d+1}}{ \tilde{C_4}\|L\|} + 2^{d-1} C_{3} \|L\|^{d-1} \Big) + \frac{4}{ \tilde{C_{4}}\|L\|}s_X(L)     \Bigg)
  \end{split}
\end{equation}

Now, using the fact that $\vol(L)$ is a continuous function of $L$ \cite{LazarsfeldMustataConvexbodies}, we may find a constant $A_{3}$ (depending only on the compact set $\kappa$) such that $\vol(L) \leq A_{3} $ for all $L$ in $\kappa$. Using this together with the bounds $A_{1} \leq \|L \| \leq A_2$ in (\ref{finalsplitinequality}), we get 
$$  C(rL) \leq \frac{1}{r^2} \Bigg(\frac{C_{1}^{d+1}}{(d+1)!\|A_{1}\|^{d+1} } \Big(   A_{3} \frac{2^{d+1}}{ \tilde{C_4} A_{1}} + 2^{d-1} C_{3} A_{2}^{d-1} \Big) + \frac{4}{ \tilde{C_{4}} A_{1}}     \Bigg).$$

So setting
$$ C^{\prime}(D) = \Bigg(\frac{C_{1}^{d+1}}{(d+1)!\|A_{1}\|^{d+1} } \Big(   A_{3} \frac{2^{d+1}}{ \tilde{C_4} A_{1}} + 2^{d-1} C_{3} A_{2}^{d-1} \Big) + \frac{4}{ \tilde{C_{4}} A_{1}}     \Bigg),$$
and using it in (\ref{almostfinalequation}), we have
$$ | s_X(rL) - s_X(rL^\prime)| \leq \frac{1}{r^2} C^{\prime}(D) \| rL - rL^\prime \|  . $$
Using the scaling property of $s_X$ for $\QQ$-divisors (\autoref{propertiesofs}), this in turn implies,
\begin{equation} \label{almostthere}
    |s_X(L)  - s_X(L^\prime) | \leq C^{\prime}(D) \| L - L^\prime \|
\end{equation}
for any two $\QQ$-divisors $L$, $L^\prime$ in $\kappa$ such that $L^\prime - L$ is big and $\|L - L^\prime \| \leq \frac{\|L\|}{2}$. Note that $C^{\prime}(D)$ only depends on the set $\kappa$ and hence only on $D$.
 
  \medskip
  To complete the proof of \autoref{mainthm}, we need to remove the assumption that $L^{\prime}-L$ is big from inequality (\ref{almostthere}). For any $L$ in $\cC$, let $\lambda_{i}(L)$ denote the $i^{\text{th}}$-coordinate of $L$ with respect to the basis $\{e_{i}\}$ (for $ 1 \leq i \leq \rho$).  Now, since $D$ is contained in the interior of $\kappa$, by \autoref{choosingr} there exists a positive $r(D)$ satisfying the
  following three conditions:
  
  \begin{enumerate}
    \item $ r(D) < \frac{A_{1}}{2}$.
      \item The closed ball
  $$ B_{r(D)} := \{D^{\prime} \, | \, \| D^{\prime} - D\| \leq r(D) \} $$ 
  is contained in the interior of $\kappa$.
  
  \item For any two $\QQ$-divisors $L$ and $L^{\prime}$ in $ B_{r(D)}$, setting $\lambda = \max_{i} \{\frac{\lambda_{i}(L)}{\lambda_{i}(L^{\prime})}\}$, we have
  $$  A_{1} < \lambda \|L^{\prime} \|  < A_{2} $$
and $$ \|\lambda L^{\prime} - L \| < \frac{A_{1}}{2} .$$

  \end{enumerate}

   Fix such an $r(D)$. For any two $\QQ$-divisors $L$ and $L^{\prime}$ in $B_{r(D)}$ such that $L^{\prime}$ is not a multiple of $L$, setting $\lambda = \max_{i} \{\frac{\lambda_{i}(L)}{\lambda_{i}(L^{\prime})} \}$, then $ \lambda L^{\prime} - L$ is ample (hence, big). Indeed, recall that $e_i$'s are an ample basis for $N^1_\R(X)$ and the $j$-th coordinate of $\lambda L' - L$ is

\[\lambda \lambda_j(L') - \lambda_j(L)  = \lambda_j(L') \left(\lambda- \frac{\lambda_j(L)}{\lambda_j(L')} \right)\geq 0\]
The right hand side is non-negative since $\lambda$ is the maximum of $\lambda_i(L)/\lambda_i(L')$. Now, if $\lambda = \lambda_j(L)/\lambda_j(L')$ for all $j$, then $L'$ is a multiple of $L$. Therefore, if $L^{\prime}$ is not a multiple of $L$, one of the coefficients of $\lambda L' - L$ is strictly positive, which implies $\lambda L' - L$ is ample.
   
   Furthermore, $\lambda L^{\prime} \in \kappa$ and $\|\lambda L^{\prime} - L \| < \frac{\|L\|}{2} $ (these are ensured by condition (c) on $r(D)$). Hence, using (\ref{almostthere}) we have 
  $$|s_X(\lambda L^{\prime}) - s_X(L)| \leq C^{\prime}(D) \| \lambda L^{\prime} - L \| $$
  for any two ample $\QQ$-divisors $L$ and $L^{\prime}$ contained in $B_{r(D)}$.

  Pick a positive constant $A_{4}$ (depending only on $D$ and $r(D)$) such that we $ \lambda_{i}(L)  \geq A_{4}$ for any $L$ in $B_{r(D)}$ and all $i$. This is possible because $\kappa$, hence the closed ball $B_{r(D)}$ is contained in the interior of the cone $\cC$. Since for some $i$, we have $\lambda = \frac{\lambda_{i}(L)}{\lambda_{i}(L^{\prime})} $, we have
  $$  |\lambda - 1| \leq \frac{|\lambda_{i} - \lambda_{i} ^{\prime} |}{\lambda_{i}^{\prime}} \leq \frac{\|L - L^{\prime}\| }{A_{4}}      .$$
  Similarly, we have
  $$  \left|\frac{1}{\lambda} - 1\right| \leq  \frac{\|L - L^{\prime}\| }{A_{4}}      .$$
  To conclude the argument, we note that
$$
\begin{aligned}
| s_X(L) - s_X(L^\prime)| \leq & |s_X(L) - s_X(\lambda L^{\prime})| + |s_X(\lambda L^{\prime}) - s_X(L^\prime)|   \\
& \leq C^{\prime}(D) \| L - \lambda L^{\prime} \|+ |\frac{1}{\lambda} - 1|s_X(L^\prime) \\
& \leq C^{\prime}(D) \|L - L^\prime \| + C^{\prime}(D)|1- \lambda| \| L^{\prime} \| + |\frac{1}{\lambda} - 1|s_X(L^{\prime}) \\
& \leq C^{\prime}(D) \|L - L^\prime \| + C^{\prime}(D) \frac{A_{2}}{A_{4}} \|L - L^{\prime} \|  + \frac{1}{A_{4}} \|L - L^{\prime}\|.
\end{aligned}   $$

Lastly, if $L^{\prime}$ were a multiple of $L$, then only the last term in the above inequality suffices. Thus, we see that for our choice of $r(D)$, choosing $C(D) = C^{\prime}(D) + C^{\prime}(D) \frac{A_{2}}{A_{4}} + \frac{1}{A_{4}}$ works for the inequlaity (\ref{Lipschitzinequality}), hence proving \autoref{mainthm}.
\qedsymbol

\section{Extending the $F$-signature function to the Nef Cone.}
 In this section, we will prove that the $F$-signature function, originally defined in Section 3 only on the ample cone (\autoref{F-sigdfn}) extends continuously to the non-zero classes in the nef cone.

\begin{thm} \label{extensionthm}
Suppose that $X$ is a globally $F$-regular projective variety of dimension $d$. Then the $F$-signature function $s_X$ extends continuously to all non-zero classes of the Nef cone $\Nef_{\RR}(X)$. Moreover, if $D$ is a nef Cartier divisor which is not big, then $s_X(D) = 0$.
\end{thm}

We prove \autoref{extensionthm} in two parts, depending on whether or not $L$ is big. First, we have the following comparison of the $F$-signature function with the volume function:

\begin{lem} \label{boundonFsig}
   Let $X$ be a globally $F$-regular projective variety of dimension $d$. Fix a norm $\| \, \|$ on the N\'eron-Severi space of $X$. Let  $C_{1}$ be a constant such that for any non-zero effective divisor $L$, we have (see \autoref{Iedfn} for the notation),
    $$ I_{e}(mL) = H^{0}(mL) \textrm{ for all $m > \frac{C_{1}}{\|L\|} p^{e}$}.$$
    
    The existence of such a constant is guaranteed by \autoref{lem.Biglowerdim}. Then, for any ample Cartier divisor $D$ on $X$, we have
   \begin{equation}
       s_X(D) \leq \frac{C_{1}^{d+1}\vol(D)}{\| D \|^{d+1} (d+1)!}.
   \end{equation}
\end{lem}
Note that the right-hand side has the same order of decay as the $F$-signature function, decaying in the order of $1/\|D\|$ as the norm of divisor $\|D\| \to \infty$.
\begin{proof}
Using \autoref{formulaforFsig} to calculate the $F$-signature $s_X(D)$, we have

\[
\begin{split}
s_X(D) &= \lim_{e\to\infty} \frac{1}{p^{e(d+1)}}\sum_{m=0}^\infty \dim_k\frac{H^0(mD)}{I_e(mD)}\\
 & = \lim_{e\to\infty} \frac{1}{p^{e(d+1)}}\sum_{m=0}^{\frac{C_{1} p^e}{\|D\|}} \dim_k\frac{H^0(mD)}{I_e(mD)} \\
 & \leq \lim_{e\to\infty} \frac{1}{p^{e(d+1)}}\sum_{m=0}^{\frac{C_{1} p^e}{\|D\|}} \dim_k H^0(mD)\\
 & \leq \lim_{e\to\infty} \frac{1}{p^{e(d+1)}} \frac{\vol(D)}{(d+1)!} \left(\frac{C_{1}p^e}{\|D\|} \right)^{d+1} \quad \text{(using the Hilbert-polynomial of $D$)} \\
 & = \frac{C_{1}^{d+1}\vol(D)}{\| D \|^{d+1} (d+1)!}
\end{split}
\]

\end{proof}

\begin{proof}[Proof of \autoref{extensionthm}]
    First suppose that $D$ is a non-zero nef divisor that is not big. Then, for any sequence $\{L_t\}_t $ of ample $\QQ$-divisors approaching $D$, choose a positive integer $r_t$ for each $t \geq 1$ such that $r_t L_t$ is integral Cartier. Then, we see that using \autoref{boundonFsig}, 
    $$s_X(L_t) = r_t \, s_X(r_t L_t) \leq  \frac{C_{1}^{d+1}\vol(r_t L_t)}{\| r_t L_t \|^{d+1} (d+1)!} = \frac{C_{1}^{d+1}\vol(L_t)}{\| L_t \|^{d+1} (d+1)!}. $$
    Since $\|D \| \neq 0$, we have that $\| L_t \|$ approaches a non-zero number (namely, $ \| D \|$) and $\vol(L_t)$ approaches $0$ as $t \to \infty$ (since $D$ is not big), this shows that $s_X(L_t) \to 0$ as $t \to \infty$. As the sequence $\{L_t\}$ chosen was arbitrary, this shows that the $F$-signature function $s_X$ extends continuously by zero to all non-zero nef divisors $L$ that are not big.
    
Now suppose that $D$ is a big and nef divisor. Following the proof of \autoref{mainthm}, to prove that $s_X$ is locally Lipschitz for ample divisors around $D$, by \autoref{normchange}, we may a pick a suitable norm depending on $D$. Since the big cone $\mathrm{Big} (X)$ is an open subset of $N^1 _\RR(X)$, given $D$ in $\mathrm{Big}(X)$, we may pick a basis $e_{1}, \dots, e_\rho$ for $N^1 _\RR  (X)$ such that each $e_i$ is the class of a big invertible sheaf and such that $D$ in contained in the interior of the cone generated by the $e_i$'s (equivalently, $D = \sum a_{i} e_i$ with each $a_i > 0$). Let $\cC = \{ a_{i} e_i \, | \, a_i \geq 0 \}$ denote the closed cone generated by the $e_i$'s and $\| \, \|$ denote the sup-norm with respect to the basis $\{e_i\}$. Then, arguing verbatim as in the final step of the proof of \autoref{mainthm} and applying the argument to all ample $\QQ$-divisors $L, L^{\prime}$ contained in $\cC$, we get positive numbers $r(D)$ and $C(D)$ such that
$$ | s_X(L) - s_X(L^{\prime}) | \leq C(D) \| L - L^{\prime}  \|$$
for any two ample $\QQ$-divisors $L$ and $L^{\prime}$ contained in a ball of radius $r(D)$ around $D$. This proves that $s_X$ is uniformly continuous in a neighbourhood of $D$, which gives us a unique continuous extension of $s_X$ to $D$.
\end{proof}

The $F$-signature function of the blow-up of $\P^2$ at a point provides an instructive example of the behavior of the function on the boundary. For a formula for general Hirzebruch surfaces, see \cite{HiroseSawadaFsigHirzebruchSurface}.

\begin{eg}\label{ToricExamples}

Let $X=\mathrm{Bl}_x(\bP^2)$ be the blow-up of $\bP^2$ at $x = [0:0:1]$. Let $H$ denote the pull-back of a line in $\bP^2$ passing through $x$ and $E$ be the exceptional divisor for the blow-up. Then $H$ and $E$ form a basis for the N\'eron-Severi space and the nef cone of $X$ is given by the divisors $aH - bE$ such that $0 \leq b \leq a$. For $L = aH-bE$, we can compute the $F$-signature of $L$ using the formula described in \cite{VonKorffthesis}, and it is given by

    \[
s_X(L) = \begin{cases}\frac{a-b}{ab}, & \text{ if }  b\leq a\leq \frac{3}{2}b \\ \frac{2b-a}{2a(a-b)} + \frac{(3b-a)(2a-3b)}{6b(a-b)^2} + \frac{(2a-3b)^2}{2a(a-b)^2} , & \text{ if } \frac{3}{2}b \leq a \leq 2b  \\
\frac{1}{a}-\frac{b^3 + (a-2b)^3}{6ab(a-b)^2}  & \text{ if } 2b\leq a\leq 3b\\
\frac{1}{a} - \frac{b^2 + (a-2b)^2 + (a-3b)(a-2b)+ (a-3b)^2}{6a(a-b)^2}& \text{ if } 3b\leq a
\end{cases}
\]

\begin{figure}[h]
\centering
\includegraphics[width=0.6\textwidth]{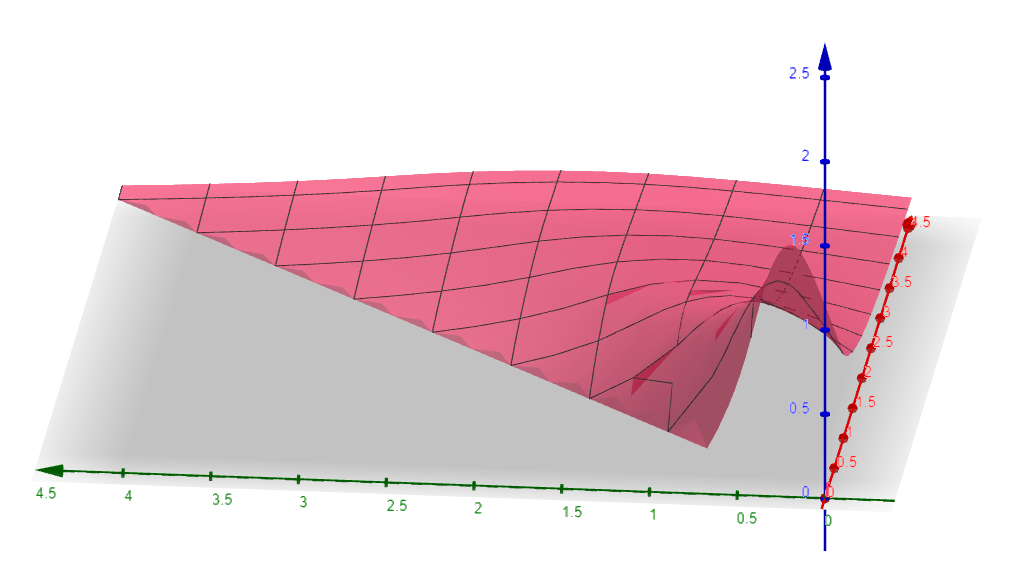}
\caption{The $F$-signature function of the blow up of $\PP^2$ at a point.}
\label{P2blowup1pt}
\end{figure}

 Note that along the line $a=b$, which corresponds to a nef but not big boundary face, the $F$-signature extends to the zero function (as proved in \autoref{extensionthm}). On the other hand, along $b=0$, which is the big and nef boundary face, letting $b\to 0$ yields $s_X(L) = \frac{1}{a}-\frac{a^2+a^2+a^2}{6a^2}=\frac{1}{2a}.$ It turns out that this corresponds to the $F$-signature of the cone over the pair $(\bP^2, \frak{m}_x)$ with respect to the divisor $aL$ on $\bP^2$ (see \cite[Theorem 4.20]{BlickleSchwedeTuckerFSigPairs1} for the definition of the $F$-signature of pairs).

\end{eg}

\begin{rem}
    \autoref{extensionthm} gives us a unique extension of the $F$-signature function to the non-zero classes in the nef cone of $X$. Further, we also know that for nef divisors that are not big, the extension is $0$. Thus, it is natural to ask what the extension to a big and nef divisor is. In forthcoming work, we explore this question and provide some answers in terms of $F$-signature of pairs, as indicated by \autoref{ToricExamples}.
\end{rem}

In particular, we can ask if the extension of the $F$-signature function to all big and nef divisors is positive. Motivated by this (and \autoref{boundonFsig}), we raise the following question on lower bounds for the $F$-signature function: 

\begin{question}\label{Question.lowerbound}
Let $X$ be a globally $F$-regular projective variety and $\| \, \|$ be a fixed norm on $N^1(X)$. Then, does there exist a constant $C>0$ (depending only on $X$) such that, we have
$$ s_{X}(L) \geq \frac{C \vol(L)}{ \|L\|^{d+1}} $$
for all ample $\QQ$-divisors $L$?
\end{question}



\section{Local upper bounds for the $F$-signature function}
In this section, we prove effective local upper bounds for the $F$-signature function (\autoref{F-sigdfn}).

\begin{thm} \label{localbounds}
Let $X$ be a globally $F$-regular projective variety. Let $d=\dim X $ be positive. Fix a basis $e_1, \dots, e_\rho$ for the N\'eron-Severi space $N^1 _\RR (X)$ such that each $e_i$ corresponds to the class of an ample and globally generated invertible sheaf. Let $\cC$ denote the simplicial cone generated by the $e_{i}$'s, that is, $\cC= \{ \sum a_{i} e_{i} \,  | \, a_{i} \in \RR_{\geq 0} \} $. Let $\| \, \|$ denote the sup-norm on $N^1 _{\RR} (X)$ with respect to the $e_{i}$'s.  Then, for any non-zero class $L$ in $\cC$, we have
\begin{equation} \label{localboundequation}  
s_X(L) \leq \frac{(d^2+2d)^{d+1} \vol(L)}{\lfloor \|L\|\rfloor ^{d+1} (d+1)!}. \end{equation}
\end{thm}

\begin{lem}
\label{lem.lowerdim}
Suppose $L$ is a globally generated ample divisor and $H$ any nef divisor on $X$. Then, for all $e \geq 1$, we have:
\begin{enumerate}
    \item $$ I_{e}(mL) = H^{0}(mL) \textrm{ for $m > (d^2+d) p^{e}$}, $$
    \item $$ I_{e} (m(nL+H)) = H^{0}(m(nL+H)) \textrm{ for all $m > \frac{(d^2 + 2d) p^{e}}{n}$}. $$
\end{enumerate}
\end{lem}

\begin{proof}
Let $S$ be the section ring of $X$ with respect to $L$. And for any $j \geq 0$, let $M^j$ be the $S$-module $\bigoplus _{t \geq 0} \cO_{X}(jH + tL)$.

\begin{enumerate}
    \item First, we claim that $S$ is generated as a graded ring by homogeneous elements of degree at most $d$. This follows from Mumford's Theorem \cite[Theorem 1.8.5]{LazarsfeldPositivity1}, if we show that the trivial bundle $\cO_{X}$ is $d$-regular with respect to $L$. Since $X$ is globally $F$-regular and $L$ is ample, by \autoref{vanishingforGFR}, we have that
    $$ H^{i} (X, \cO_{X}((d-i )L)) = 0 \quad \text{for all $i >0$}    . $$
    This implies that $\cO_{X}$ is $d$-regular with respect to $L$ and hence that $S$ is generated by elements of degree at most $d$. 
    
    
     Since the section ring $S$ is generated by elements of degree $\leq d$, the homogeneous maximal ideal $\mathfrak{m}=S_{>0}$ is generated in degrees $\leq d$. By \cite[Proposition 8.3.8]{HunekeSwansonIntegralClosure}, there exist elements $x_0, \dots, x_d$ (not necessarily homogeneous), such that all terms of each $x_i$ have degree at most $d$, and the integral closure  $\overline{(x_0,\dots, x_d)}$ is equal to the maximal ideal $\fm$. Now, by using the Brian\c con-Skoda theorem in the strongly $F$-regular ring $S$ \cite[Theorem 5.4]{HochsterHunekeTC1}, we have
    \[\frak{m}^{(d+1)p^e}=\overline{(x_0^{p^e}, \dots, x_d^{p^e})^{d+1}} \subseteq (x_0^{p^e}, \dots, x_d^{p^e}).\]
    
    Therefore, if $m\geq d(d+1)p^e$, for any element $x\in S_m=H^0(mL)$, by the pigeon-hole principle, we have $x \in \fm^{(d+1)p^e}$, and consequently, $x \in (x_0 ^{p^e}, \dots, x_d ^{p^e})$. Hence, the map $\cO_X\to F^e_*\cO_X(mL)$ sending $1\mapsto F^e_*x$ cannot split.
    
    \item Similarly as in part (a), we claim that for any $j \geq 0$, $M^j$ is generated over $S$ by elements of degree at most $d$. For this, again by Mumford's theorem, it is enough to show that $\cO_{X}(jH)$ is $d$-regular with respect to $L$. Since $H$ is nef, by \autoref{vanishingforGFR} we again have:
    $$ H^{i} (X, \cO_{X}(jH + (d-i )L)) = 0 \quad \text{for all $i >0$}    . $$
   Suppose $f \in  H^{0}(m(nL+H) \setminus I_{e} (m(nL+H))$ and $m > \frac{(d^2 + 2d) p^{e}}{n}$, then we may write $f = \sum r_{i} f_{i}$ for $r_{i} \in S$ and $f_{i} \in M^{m}$ with degree of $f_{i}$ at most $d$. Then the degree of each $r_{i}$ is at least $(d^2 + d) p^{e}$. Now, since by assumption, the map $\cO_{X} \to \Fe \cO_{X}(m(nL+H))$ sending $1$ to $\Fe f$ splits, we must have that for some $i$, $r_{i} \in H^{0}(kL) \setminus I_{e}(kL)$ for a suitable $k > (d^2 + d)p^{e}$, contradicting part (a) of the lemma. Hence, we must have $I_{e} (m(nL+H)) = H^{0}(m(nL+H))$ for all $m > \frac{(d^2 + 2d) p^{e}}{n}$. This completes the proof of the lemma. \qedhere \end{enumerate}\end{proof}

\begin{lem} \label{uniformeffectiveconstants}
    Fix a basis $e_1, \dots, e_\rho$ of $N^1 _\RR (X)$ such that each $e_i$ corresponds to the class of an ample and globally generated invertible sheaf. Let $\cC$ denote the simplicial cone generated by the $e_{i}$'s, that is, $\cC= \{ \sum a_{i} e_{i} \,  | \, a_{i} \in \RR_{\geq 0} \} $. Let $\| \, \|$ denote the sup-norm on $N^1 _{\RR} (X)$ with respect to the $e_{i}$'s. Then, for any invertible sheaf $\cL$ such that its class $L$ in the N\'{e}ron-Severi space satisfies $L \in \cC$ and $\|L \| \geq d$ (where $d$ is the dimension of $X$), we have
    \begin{itemize}
        \item [(a)] $\cL$ is ample and globally generated.

        \item [(b)] Further,
        $$ I_{e}(mL) = H^{0}(mL) \textrm{ for all $m > \frac{(d^2+2d) p^{e}}{\lfloor \|L\| \rfloor}$}. $$
    \end{itemize}
     
\end{lem}

\begin{proof}
    \begin{enumerate}
        \item Ampleness of $\cL$ follows from the assumption that $L$ lies in $\cC$ and $L$ is non-zero since $\| L \| \neq 0$. It remains to show global generation of $\cL$. For this, we note that since $\|L\| \geq d$, there is some $i$ such that we may decompose the divisor $L$ as $L = dL_{i} + H$ where $H$ is some nef Cartier divisor and $L_i$ is a Cartier divisor corresponding to the class $e_i$. This follows from the assumption that $e_i$'s  are integral, ample and globally generated and the fact that the sup-norm is achieved by some coordinate of $L$. 
        Hence, applying \autoref{vanishingforGFR}, we have
$$ H^{p} (X, \cO_{X}(  L -  p \, L_i)) = 0  \quad \text{for all $p > 0$.} $$
Therefore, $\cL$ is $0$-regular with respect to the globally generated ample divisor $L_i$. Hence, $\cL$ is globally generated itself.

        \item Since $\|L\| \geq d$, for some $0 \leq i \leq \rho$, we may write $L = \lfloor \|L\| \rfloor e_{i} + H$ for some integral and nef class $H$. Now, applying Part (b) of \autoref{lem.lowerdim}, we get
        $$ I_{e}(mL) = H^{0}(mL) \textrm{ for all $m > \frac{(d^2+2d) p^{e}}{\lfloor \|L\| \rfloor}$}. $$

    \end{enumerate} 
\end{proof}

\begin{proof}[Proof of \autoref{localbounds}: ]
By \autoref{mainthm}, the $F$-signature function is continuous, hence we may prove \autoref{localbounds} only when $L$ is an ample $\QQ$-divisor. Further, since both sides of (\ref{localboundequation}) scale inverse-linearly, we may assume that $L$ is a Cartier divisor and $\|L\| \geq d$. Then, applying Part (b) of \autoref{uniformeffectiveconstants}, the Theorem follows from \autoref{boundonFsig} by using $\frac{d^2 + 2d}{\lfloor \|L\| \rfloor}$ instead of $\frac{C_{1}}{\|L\|}$.
\end{proof}

\vskip 12pt

{\bf Acknowledgements:} We would like to thank Kevin Tucker for numerous useful discussions and suggestions for this paper when they were visiting Chicago. We would also like to thank our advisors Karl Schwede and Karen Smith for encouraging us to start this paper, for many valuable discussions, and for providing support for this project. We thank Harold Blum, Christopher Hacon, and Mircea Musta\c t\u a for valuable suggestions and discussions. We would also like to thank Anna Brosowsky, Alapan Mukhopdhyay, Devlin Mallory, Qingyuan Xue, and Jos\'e Y\'a\~nez for many helpful discussions. We thank the referees for a careful reading of the paper and providing numerous suggestions that have greatly improved the paper. Lastly, we thank the mathematics departments at the University of Utah, the University of Illinois, Chicago, and the University of Michigan for their hospitality and support during the various visits by the authors.

\bibliographystyle{skalpha}
\bibliography{MainBib}

\newcommand{\etalchar}[1]{$^{#1}$}
\def\cfudot#1{\ifmmode\setbox7\hbox{$\accent"5E#1$}\else
  \setbox7\hbox{\accent"5E#1}\penalty 10000\relax\fi\raise 1\ht7
  \hbox{\raise.1ex\hbox to 1\wd7{\hss.\hss}}\penalty 10000 \hskip-1\wd7\penalty
  10000\box7}
\providecommand{\bysame}{\leavevmode\hbox to3em{\hrulefill}\thinspace}
\providecommand{\MR}{\relax\ifhmode\unskip\space\fi MR}
\providecommand{\MRhref}[2]{%
  \href{http://www.ams.org/mathscinet-getitem?mr=#1}{#2}
}
\providecommand{\href}[2]{#2}
\begin{thebibliography}{ELM{\etalchar{+}}05}

\bibitem[AL03]{AberbachLeuschke}
{\sc I.~M. Aberbach and G.~J. Leuschke}: \emph{The {$F$}-signature and strong
  {$F$}-regularity}, Math. Res. Lett. \textbf{10} (2003), no.~1, 51--56.
  {\sf\scriptsize MR1960123 (2004b:13003)}

\bibitem[BST12]{BlickleSchwedeTuckerFSigPairs1}
{\sc M.~Blickle, K.~Schwede, and K.~Tucker}: \emph{{$F$}-signature of pairs and
  the asymptotic behavior of {F}robenius splittings}, Adv. Math. \textbf{231}
  (2012), no.~6, 3232--3258. {\sf\scriptsize 2980498}

\bibitem[Bou02]{BoucksomVolumeofLineBundle}
{\sc S.~Boucksom}: \emph{On the volume of a line bundle}, Internat. J. Math.
  \textbf{13} (2002), no.~10, 1043--1063. {\sf\scriptsize 1945706}

\bibitem[CRST18]{carvajal-rojas_fundamental_2016}
{\sc J.~Carvajal-Rojas, K.~Schwede, and K.~Tucker}: \emph{Fundamental groups of
  ${F}$-regular singularities via ${F}$-signature}, Ann. Sci. \'{E}c. Norm.
  Sup\'{e}r. (4) \textbf{51} (2018), no.~4, 993--1016.

\bibitem[CR22]{CarvajalRojasFiniteTorsors}
{\sc J.~A. Carvajal-Rojas}: \emph{Finite torsors over strongly {$F$}-regular
  singularities}, \'{E}pijournal G\'{e}om. Alg\'{e}brique \textbf{6} (2022),
  Art. 1, 30. {\sf\scriptsize 4391081}

\bibitem[DSPY22]{DeStefaniPolstraYaoGlobalFsplittingRatio}
{\sc A.~De~Stefani, T.~Polstra, and Y.~Yao}: \emph{Global {F}-splitting ratio
  of modules}, J. Algebra \textbf{610} (2022), 773--792. {\sf\scriptsize
  4470700}

\bibitem[ELM{\etalchar{+}}05]{EinLazMusNakPopAsymptoticInvariantsofLineBundles}
{\sc L.~Ein, R.~Lazarsfeld, M.~Musta\c{t}\v{a}, M.~Nakamaye, and M.~Popa}:
  \emph{Asymptotic invariants of line bundles}, Pure Appl. Math. Q. \textbf{1}
  (2005), no.~2, Special Issue: In memory of Armand Borel. Part 1, 379--403.
  {\sf\scriptsize 2194730}

\bibitem[Fol99]{FollandRealAnalysis}
{\sc G.~B. Folland}: \emph{Real analysis}, second ed., Pure and Applied
  Mathematics (New York), John Wiley \& Sons, Inc., New York, 1999, Modern
  techniques and their applications, A Wiley-Interscience Publication.
  {\sf\scriptsize 1681462}

\bibitem[GLP{\etalchar{+}}15]{gongyo_rational_2015}
{\sc Y.~Gongyo, Z.~Li, Z.~Patakfalvi, K.~Schwede, H.~Tanaka, and R.~Zong}:
  \emph{On rational connectedness of globally {$F$}-regular threefolds}, Adv.
  Math. \textbf{280} (2015), 47--78.

\bibitem[GOST15]{GongyoSingularitiesofcoxrings}
{\sc Y.~Gongyo, S.~Okawa, A.~Sannai, and S.~Takagi}: \emph{Characterization of
  varieties of fano type via singularities of cox rings}, J. Algebraic Geom.
  \textbf{24} (2015), no.~1, 159--182.

\bibitem[GT19]{GongyoTakagiInjectivityThmforGFR}
{\sc Y.~Gongyo and S.~Takagi}: \emph{Koll\'{a}r's injectivity theorem for
  globally {$F$}-regular varieties}, Eur. J. Math. \textbf{5} (2019), no.~3,
  872--880. {\sf\scriptsize 3993268}

\bibitem[HM06]{HaconMckernanBoundednessofpluricanonicalmaps}
{\sc C.~D. Hacon and J.~McKernan}: \emph{Boundedness of pluricanonical maps of
  varieties of general type}, Invent. Math. \textbf{166} (2006), no.~1, 1--25.
  {\sf\scriptsize 2242631}

\bibitem[HX15]{HaconXuThreeDimensionalMinimalModel}
{\sc C.~D. Hacon and C.~Xu}: \emph{On the three dimensional minimal model
  program in positive characteristic}, J. Amer. Math. Soc. \textbf{28} (2015),
  no.~3, 711--744.

\bibitem[Har77]{Hartshorne}
{\sc R.~Hartshorne}: \emph{Algebraic geometry}, Springer-Verlag, New York,
  1977, Graduate Texts in Mathematics, No. 52. {\sf\scriptsize MR0463157 (57
  \#3116)}

\bibitem[HS17]{HiroseSawadaFsigHirzebruchSurface}
{\sc D.~Hirose and T.~Sawada}: \emph{Korff {F}-signatures of {H}irzebruch
  surfaces}, \url{https://arxiv.org/abs/1701.01905} (2017).

\bibitem[HH89]{HochsterHunekeTightClosureAndStrongFRegularity}
{\sc M.~Hochster and C.~Huneke}: \emph{Tight closure and strong
  {$F$}-regularity}, M\'em. Soc. Math. France (N.S.) (1989), no.~38, 119--133,
  Colloque en l'honneur de Pierre Samuel (Orsay, 1987). {\sf\scriptsize
  MR1044348 (91i:13025)}

\bibitem[HH90]{HochsterHunekeTC1}
{\sc M.~Hochster and C.~Huneke}: \emph{Tight closure, invariant theory, and the
  {B}rian\c con-{S}koda theorem}, J. Amer. Math. Soc. \textbf{3} (1990), no.~1,
  31--116. {\sf\scriptsize MR1017784 (91g:13010)}

\bibitem[HL02]{HunekeLeuschkeTwoTheoremsAboutMaximal}
{\sc C.~Huneke and G.~J. Leuschke}: \emph{Two theorems about maximal
  {C}ohen-{M}acaulay modules}, Math. Ann. \textbf{324} (2002), no.~2, 391--404.
  {\sf\scriptsize MR1933863 (2003j:13011)}

\bibitem[HS06]{HunekeSwansonIntegralClosure}
{\sc C.~Huneke and I.~Swanson}: \emph{Integral closure of ideals, rings, and
  modules}, London Mathematical Society Lecture Note Series, vol. 336,
  Cambridge University Press, Cambridge, 2006. {\sf\scriptsize MR2266432
  (2008m:13013)}

\bibitem[Kaw21]{KawakamiBogomolovSommese}
{\sc T.~Kawakami}: \emph{Bogomolov-{S}ommese type vanishing for globally
  {$F$}-regular threefolds}, Math. Z. \textbf{299} (2021), no.~3-4, 1821--1835.
  {\sf\scriptsize 4329270}

\bibitem[Kle66]{Kleimannumericaltheory}
{\sc S.~L. Kleiman}: \emph{Toward a numerical theory of ampleness}, Ann. of
  Math. (2) \textbf{84} (1966), 293--344. {\sf\scriptsize 206009}

\bibitem[K\"06]{KuronyaAsymptoticcohomologicalfunctions}
{\sc A.~K\"{u}ronya}: \emph{Asymptotic cohomological functions on projective
  varieties}, Amer. J. Math. \textbf{128} (2006), no.~6, 1475--1519.
  {\sf\scriptsize 2275909}

\bibitem[Laz04]{LazarsfeldPositivity1}
{\sc R.~Lazarsfeld}: \emph{Positivity in algebraic geometry. {I}}, Ergebnisse
  der Mathematik und ihrer Grenzgebiete. 3. Folge. A Series of Modern Surveys
  in Mathematics [Results in Mathematics and Related Areas. 3rd Series. A
  Series of Modern Surveys in Mathematics], vol.~48, Springer-Verlag, Berlin,
  2004, Classical setting: line bundles and linear series. {\sf\scriptsize
  MR2095471 (2005k:14001a)}

\bibitem[LM09]{LazarsfeldMustataConvexbodies}
{\sc R.~Lazarsfeld and M.~Musta\c{t}\u{a}}: \emph{Convex bodies associated to
  linear series}, Ann. Sci. \'{E}c. Norm. Sup\'{e}r. (4) \textbf{42} (2009),
  no.~5, 783--835. {\sf\scriptsize 2571958}

\bibitem[LLX20]{LiLiuXuAGuidedTour}
{\sc C.~Li, Y.~Liu, and C.~Xu}: \emph{A guided tour to normalized volume},
  Geometric analysis---in honor of {G}ang {T}ian's 60th birthday, Progr. Math.,
  vol. 333, Birkh\"{a}user/Springer, Cham, [2020] \copyright 2020,
  pp.~167--219. {\sf\scriptsize 4181002}

\bibitem[MPST19]{MaPolstraSchwedeTuckerFsigBirational}
{\sc L.~Ma, T.~Polstra, K.~Schwede, and K.~Tucker}: \emph{{$F$}-signature under
  birational morphisms}, Forum Math. Sigma \textbf{7} (2019), Paper No. e11,
  20. {\sf\scriptsize 3940231}

\bibitem[Mar22]{MartinTorsionDivisors}
{\sc I.~Martin}: \emph{The number of torsion divisors in a strongly
  {$F$}-regular ring is bounded by the reciprocal of {$F$}-signature}, Comm.
  Algebra \textbf{50} (2022), no.~4, 1595--1605. {\sf\scriptsize 4391510}

\bibitem[Pol22]{PolstramaximalCM}
{\sc T.~Polstra}: \emph{A theorem about maximal {C}ohen-{M}acaulay modules},
  Int. Math. Res. Not. IMRN (2022), no.~3, 2086--2094. {\sf\scriptsize 4373232}

\bibitem[SS10]{SchwedeSmithLogFanoVsGloballyFRegular}
{\sc K.~Schwede and K.~E. Smith}: \emph{Globally {$F$}-regular and log {F}ano
  varieties}, Adv. Math. \textbf{224} (2010), no.~3, 863--894. {\sf\scriptsize
  2628797 (2011e:14076)}

\bibitem[Smi97]{SmithVanishingSingularitiesAndEffectiveBounds}
{\sc K.~E. Smith}: \emph{Vanishing, singularities and effective bounds via
  prime characteristic local algebra}, Algebraic geometry---{S}anta {C}ruz
  1995, Proc. Sympos. Pure Math., vol.~62, Amer. Math. Soc., Providence, RI,
  1997, pp.~289--325. {\sf\scriptsize MR1492526 (99a:14026)}

\bibitem[Smi00]{SmithGloballyFRegular}
{\sc K.~E. Smith}: \emph{Globally {F}-regular varieties: applications to
  vanishing theorems for quotients of {F}ano varieties}, Michigan Math. J.
  \textbf{48} (2000), 553--572, Dedicated to William Fulton on the occasion of
  his 60th birthday. {\sf\scriptsize MR1786505 (2001k:13007)}

\bibitem[SVdB97]{SmithVanDenBerghSimplicityOfDiff}
{\sc K.~E. Smith and M.~Van~den Bergh}: \emph{Simplicity of rings of
  differential operators in prime characteristic}, Proc. London Math. Soc. (3)
  \textbf{75} (1997), no.~1, 32--62. {\sf\scriptsize MR1444312 (98d:16039)}

\bibitem[Tak06]{TakayamaPluricanonicalsystems}
{\sc S.~Takayama}: \emph{Pluricanonical systems on algebraic varieties of
  general type}, Invent. Math. \textbf{165} (2006), no.~3, 551--587.
  {\sf\scriptsize 2242627}

\bibitem[Tay19]{TaylorInvAdjFsig}
{\sc G.~Taylor}: \emph{Inversion of adjunction for $f$-signature}, 2019.

\bibitem[Tuc12]{TuckerFSigExists}
{\sc K.~Tucker}: \emph{{$F$}-signature exists}, To appear in Inventiones
  Mathematicae, arXiv:1103.4173.

\bibitem[VK12]{VonKorffthesis}
{\sc M.~R. Von~Korff}: \emph{The {F}-{S}ignature of {T}oric {V}arieties},
  ProQuest LLC, Ann Arbor, MI, 2012, Thesis (Ph.D.)--University of Michigan.
  {\sf\scriptsize 3093997}

\end{thebibliography}

\end{document}